\documentclass[reqno,11pt]{amsart}
\usepackage[foot]{amsaddr}
\usepackage{amsmath,mathrsfs,amsthm,amstext,amscd,url,amssymb,mathtools,faktor}
\usepackage{amsfonts}
\usepackage{enumerate}
\usepackage{comment}
\usepackage{braket}
\usepackage[export]{adjustbox}
\usepackage{wasysym}
\usepackage{acronym}
\usepackage{paralist}
\usepackage{xspace}
\usepackage{hyperref}
\usepackage{scrextend}
\usepackage{latexsym,textcomp,multirow}
\usepackage{graphicx}
\usepackage{textcomp}
\usepackage{empheq}
\usepackage{siunitx}
\usepackage{wrapfig}
\usepackage{chngpage}
\usepackage{graphics,subfigure}
\usepackage{tabularx}
\usepackage{booktabs}
\usepackage{epigraph}
\usepackage[inline]{enumitem}
\usepackage{appendix}
\usepackage{tikz}
\allowdisplaybreaks

\newtheorem{theorem}{Theorem}
\newtheorem{lemma}{Lemma}

\newtheorem{proposition}{Proposition}
\newtheorem{definition}{Definition}
\newtheorem{remark}{Remark}

\usepackage{palatino}
\usepackage[T1]{fontenc}
\usepackage[mathscr]{eucal}

\usepackage{hyperref}

\numberwithin{equation}{section}  
\usepackage[sort&compress,capitalize,nameinlink]{cleveref}
\crefname{example}{Ex.}{Exs.}
\crefname{app}{Appendix}{Appendices}
\crefname{assumption}{Assumption}{Assumptions}

\crefrangeformat{equation}{\upshape(#3#1#4)\textendash(#5#2#6)}

\makeatletter
\newenvironment{proofoflemma}[1]{
	\pushQED
	{\qed}%
	\normalfont \topsep6\p@\@plus6\p@\relax
	\trivlist
	\item[\hskip\labelsep
	{\noindent{\textit{Proof of Lemma #1}}}. \hspace{0.1cm}]\ignorespaces
}{{\begin{flushright}$\qed$
	\end{flushright}}
}
\makeatother



\usepackage[textwidth=30mm]{todonotes}

\usepackage{soul}
\setstcolor{red}
\sethlcolor{SkyBlue}

\newcommand{\E}{\ensuremath{\mathbb{E}}}

\newcommand{\N}{\ensuremath{\mathbb{N}}}
\newcommand{\Z}{\ensuremath{\mathbb{Z}}}

\newcommand{\R}{\ensuremath{\mathbb{R}}}

\newcommand{\bin}{\ensuremath{{\rm Bin}}}
\newcommand{\scale}{\vartheta}
\newcommand{\sscale}{\gamma^{(k)}}
\newcommand{\ssscale}{\bar{\gamma}^{(k)}}
\newcommand{\negbin}{\textrm{NegBin}}
\newcommand{\poisson}{\textrm{Poisson}}

\newcommand{\bx}{\mathbf{x}}
\newcommand{\by}{\mathbf{y}}

\newcommand{\bn}{\mathbf{n}}
\newcommand{\bm}{\mathbf{m}}

\newcommand{\bu}{\mathbf{u}}

\newcommand{\1}{\mathbf{1}}
\newcommand{\sy}{\textrm{sym}}
\renewcommand{\:}{\,\textrm{\normalfont:}\,}
\renewcommand{\Pr}{\mathbb{P}}
\renewcommand{\sp}{\textrm{\normalfont span}}
\newcommand{\oto}{\text{\normalfont one-to-one}}
\newcommand{\SEP}{\text{\normalfont SEP}}
\newcommand{\IRW}{\text{\normalfont IRW}}
\newcommand{\SIP}{\text{\normalfont SIP}}

\newcommand{\Scale}{\varTheta}
\newcommand{\dd}{\text{\normalfont d}}
\newcommand{\cpi}{\bar{\pi}}

\makeatletter
\newsavebox{\@brx}
\newcommand{\llangle}[1][]{\savebox{\@brx}{\(\m@th{#1\langle}\)}%
	\mathopen{\copy\@brx\kern-0.5\wd\@brx\usebox{\@brx}}}
\newcommand{\rrangle}[1][]{\savebox{\@brx}{\(\m@th{#1\rangle}\)}%
	\mathclose{\copy\@brx\kern-0.5\wd\@brx\usebox{\@brx}}}
\makeatother

\newmuskip\pFqmuskip
\newcommand*\pFq[6][8]{%
	\begingroup 
	\pFqmuskip=#1mu\relax
	\mathcode`\,=\string"8000
	\begingroup\lccode`\~=`\,
	\lowercase{\endgroup\let~}\pFqcomma
	{}_{#2}F_{#3}{\left[\genfrac..{0pt}{}{#4}{#5};#6\right]}%
	\endgroup
}
\newcommand{\pFqcomma}{\mskip\pFqmuskip}

\newcommand{\closure}[2][3]{%
	{}\mkern#1mu\overline{\mkern-#1mu#2}}

\makeatletter
\renewcommand{\email}[2][]{%
	\ifx\emails\@empty\relax\else{\g@addto@macro\emails{,\space}}\fi%
	\@ifnotempty{#1}{\g@addto@macro\emails{\textrm{(#1)}\space}}%
	\g@addto@macro\emails{#2}%
}
\makeatother

\title[Higher-order hydrodynamics and equilibrium fluctuations]{Higher-order hydrodynamics and equilibrium fluctuations
of interacting particle systems}
\author{Joe P.\ Chen}
\address[J.P.C.]{Department of Mathematics\\
	Colgate University\\
Hamilton, NY 13346\\
USA}
\email[J.P.C.]{jpchen@colgate.edu}
\author{Federico Sau}
\address[F.S.]{Institute of Science and Technology (IST Austria)
\\ Klosterneuburg, 3400\\
	  Austria}
  \email[F.S.]{federico.sau@ist.ac.at}
\begin{document}

\begin{abstract}
	Motivated by the recent preprint [\emph{arXiv:2004.08412}] by Ayala, Carinci, and Redig, we first provide a general framework for the study of scaling limits of higher-order fields. Then, by considering the same class of infinite interacting particle systems as in [\emph{arXiv:2004.08412}], namely symmetric simple exclusion and inclusion processes in the $d$-dimensional Euclidean lattice, we prove the hydrodynamic limit, and  convergence for the equilibrium fluctuations, of higher-order fields.  In particular, the limit fields exhibit a tensor structure. Our fluctuation result differs from that in [\emph{arXiv:2004.08412}], since we consider a different notion of higher-order fluctuation fields. 
\end{abstract}
	\maketitle
	{\footnotesize
}

	\section{Introduction}
 Within the theory of hydrodynamic limits for interacting particle systems (see e.g.\ the surveys \cite{de_masi_survey_1984, de_masi_mathematical_1991, kipnis_scaling_1999}), nearly all limit theorems, encompassing laws of large numbers as well as asymptotic analyses of fluctuations and large deviations,   concern the dynamical behavior of the empirical density of particles. In particular, for such empirical density fields  progress has been recently made in the understanding of 	 equilibrium and non-equilibrium fluctuations, including boundary dynamics or random environments, as well as considering more general geometries (see e.g.\ \cite{jara_non-equilibrium_2018, goncalves_hydrodynamics_2019, chen2019asymptotic, jara_quenched_2008, faggionato_cluster_2008, redig_symmetric_2020, landim2020hydrodynamic} and references therein).
 
 In this article, we take a  step towards higher-order fields, and study higher-order hydrodynamic limits and the corresponding equilibrium fluctuations for a class of infinite interacting particle systems. Higher-order fields---in contrast to  empirical density fields, which may be considered as first-order fields---refer to (random) empirical measures of higher  moments of the occupation variables of the particle systems; in this sense, empirical density fields are empirical measures of first moments of the occupation variables. We express	 these higher-order fields in terms of factorial moments of the occupation variables; this definition turns out to be  more  natural than using regular moments in the context of interacting particle systems.

The  motivation for the study of scaling limits of higher-order fields is twofold. On the one side, at the microscopic level,  a hierarchical structure in terms of higher-order fields provides an elegant decomposition of the Markovian particle dynamics. On the other side,  scaling limits of higher-order fields  represent a refinement of most typical hydrodynamic results (\cite{kipnis_scaling_1999}) displaying phenomena of {asymptotic independence}. Indeed, the discrete $k$th-order fields, roughly speaking, resemble $k$-fold tensor products of empirical density fields. While the particle interaction creates correlations destroying  this tensor structure, such an asymptotic independence---also  referred to in the literature  as   \emph{propagation of chaos}, although sometimes meaning slightly different concepts (see e.g.\ \cite{sznitman1991topics, kolokoltsov_nonlinear_2010, chen2017hydrodynamic} and references therein)---may or may not emerge at the scale of the hydrodynamic limit (a law of large numbers) or at the scale of the fluctuations (a central limit theorem).
In this paper we answer this question for a class of conservative symmetric interacting particle systems.

As interacting particle systems for which we prove our results, we  consider 	 infinite  systems of  exclusion (see e.g.\ \cite{liggett_interacting_2005-1}) and inclusion (see e.g.\ \cite{giardina_duality_2009})  particles, symmetrically evolving in $\Z^d$ with nearest-neighboring jumps. This last hypothesis is not crucial and it may be loosened to include symmetric finite-range jumps as e.g.\ in \cite{ayala2020higher}, but it simplifies	the notation. Along with these two examples, as a simplified instance of these interacting systems, we also consider a  system of independent particles. The main reason to study these particle systems first is the fact that, despite the particle interaction which breaks down the product structure  of the higher-order fields' evolution, the differential equations of suitably weighted factorial moments of the same order form a closed \emph{linear} system. In turn, such linearity comes as a result of the \emph{self-duality} property (\cite{liggett_interacting_2005-1}) of the  particle systems. 
The specific particle systems which we consider here are, among a larger class including e.g.\  symmetric (non-trivial) zero-range processes, the only conservative interacting particle systems with self-duality, see e.g.\ \cite[Theorem 2.1]{redig_factorized_2018}). Duality in this context asserts that, not only the expected number of particles at a location is related to the behavior of a single particle starting from that same location, but also higher-order factorial moments (if suitably normalized) at multiple locations may be expressed in terms of expectations of as many particles as the order of the moment considered. In fact,  the so-called \textquotedblleft dual\textquotedblright\ particles follow the same interaction rules of the original system, which is what gives the name \textquotedblleft self-duality\textquotedblright.   

As a consequence of the self-duality, the equations governing the evolution of the higher-order fields are linear and, moreover, no replacement lemmas (see e.g.\ \cite{kipnis_scaling_1999}) are required in order to close the corresponding equations. Nevertheless, for the purpose of deriving rigorous scaling limits,  we  first provide an expansion  of the objects determining the evolution of the fields of order $k \in \N$ in terms of higher-order fields of order at most $2k$, and, then,  reconstruct from there  the limiting evolution.   In this sense, we claim that the higher-order fields  and, in particular,  the decompositions that we present here offer a framework  which fits the study of more general \textquotedblleft non-linear\textquotedblright\ (i.e.\ for which self-duality does not hold) interacting particle systems, such as zero-range processes as well asymmetric models, for which higher-order replacement lemmas  present a main challenge.

Our main results on the $k$th-order fields can be summarized as follows.
For the hydrodynamic limit, the particle interaction vanishes on the macroscopic scale, and the corresponding hydrodynamic equation becomes the tensorization of $k$ copies of the same deterministic linear heat equation. 
This result is valid for any initial  distribution of particles which satisfies (a)    a $k$th-order weak law of large numbers and (b) a uniform bound for the factorial moments. In particular, we do not assume the initial distribution to be a local Gibbs state in product form. 
As for the equilibrium fluctuations,  we prove convergence, as $N$ goes to infinity, of the  fluctuations of order $N^{-d/2}$ to a $k$-dimensional Gaussian generalized Ornstein-Uhlenbeck process with tensorized drift and white noise with deterministic quadratic variation. Here, as in the case of  equilibrium fluctuations for the first-order empirical density  fields (\cite[\S 11]{kipnis_scaling_1999}), the particle interaction appears as mobility coefficient  in the limiting generalized Ornstein-Uhlenbeck equations.	

Our work takes inspiration from the  seminal papers \cite{assing_limit_2007, goncalves_quadratic_2019, ayala2020higher}, in which different  observables and limit theorems for higher-order fields of particle systems are considered.  More specifically, the authors of \cite{ayala2020higher}, motivated by the recent study on orthogonal polynomial dualities (see e.g.\ \cite{franceschini_stochastic_2017, redig_factorized_2018}), study the asymptotic behavior of   order $N^{-kd/2}$ equilibrium fluctuations for the $k$th-order fields, unveiling a recursive structure in the quadratic variation of the noise involving equilibrium fluctuations of order $N^{-kd/2+d/2}$. In particular, while for the case of $k=1$ corresponding to empirical density fields the two notions of equilibrium fluctuations---and, hence, the two results---coincide, they become two distinct objects  as soon as $k \geq 2$. As we explain in \cref{section:comparison} below, the difference between the orders of the fluctuations considered in this paper and in \cite{ayala2020higher} originates from the different  procedure employed to \textquotedblleft center\textquotedblright\ the higher-order fields; in particular, this consideration opens the possibility of studying fluctuations of intermediate orders. 

On the one hand, compared to the setting in \cite{assing_limit_2007, goncalves_quadratic_2019}, our higher-order fields correspond, roughly speaking, to  $k$-fold tensor products of first-order fields.
As already mentioned, the dynamics of the interacting particle  produces a  coupling effect between  the $k$ components of the $k$th-order fields. 
 On the other hand, the quadratic fields studied in \cite{assing_limit_2007, goncalves_quadratic_2019}---which work specifically with  simple exclusion processes---are defined, approximately, as squares of first-order fields.
Hence, from our second-order fields one recovers such quadratic fields  by considering only degenerate test functions suitably  approximating Dirac masses supported \textquotedblleft on the diagonal\textquotedblright. This correspondence naturally extends to higher-order fields and  powers thereof, and is facilitated through the choice of test functions supported on hyperplanes in the Euclidean space. In any case, due to the singularity of these test functions, none of the results we consider here straightforwardly translates to the context of \cite{assing_limit_2007, goncalves_quadratic_2019}, and \emph{vice versa}. 

The rest of the paper is organized as follows. 
In \cref{section:setting}, we set up the notation recurring throughout the paper; present definitions and properties of higher-order fields (\cref{section:fields}); and introduce the \textquotedblleft linear\textquotedblright\ interacting particle systems under consideration (\cref{section:infinite_particle_systems}). 
In \cref{section:main_results} we present our two main results on higher-order hydrodynamics (\cref{section:hydrodynamics}) and equilibrium fluctuations (\cref{section:fluctuations}). 
Their proofs  are detailed, respectively, in \cref{section:proof_hydrodynamics} and \cref{section:proof_fluctuations}.
Finally, \cref{appendix:infinite_particle_systems} and \cref{appendix:test_functions} contain some extra material integrating \cref{section:setting} on the construction of the infinite particle systems and a discussion on the space of test functions considered, respectively.

\section{Setting and notation}\label{section:setting}
In this section, we introduce, first,  the higher-order fields and, then,  the particle systems we consider in this paper. In both cases, we discuss and prove some of their basic properties. We emphasize that the content of \cref{section:fields} below is independent of the specific dynamics imposed on the particle system. Some of the notation we will employ all throughout the paper  is schematically presented below:  for $k, \ell \in \N$ with $\ell \leq k$, letting $\Z^d$ denote the $d$-dimensional Euclidean lattice and
\begin{equation*}
\bx=\ (x_1,\ldots, x_k) \in (\Z^d)^k\ ,\qquad
\by=\ (y_1,\ldots, y_\ell) \in (\Z^d)^\ell
\ ,
\end{equation*}
we define
\begin{align*}
[k]:=&\ \{1,\ldots, k\}\ ,\qquad [k]_0:= \{0,1,\ldots, k\}\\
\bx\:\by:=&\ (x_1,\ldots,x_k,y_1,\ldots, y_\ell ) \in (\Z^d)^{k+\ell}\\
\bx_i^{y}:=&\ (x_1,\ldots, x_{i-1},y,x_{i+1},\ldots, x_k)\\
\widehat\bx_{{\{i_1,\ldots,i_\ell\}}}:=&\ (x_1,\ldots, x_{i_1-1},x_{i_1+1},\ldots, x_{i_\ell-1},x_{i_\ell+1},\ldots, x_k)
\end{align*}
while, if $\Sigma_k$ stands for the set of permutations of $k$ indices and $\varsigma \in \Sigma_k$,
\begin{align*}
\varsigma \bx:=&\ (x_{\varsigma(1)},\ldots, x_{\varsigma(k)})\ .
\end{align*}
\subsection{Higher-order fields}\label{section:fields}
Let us start by introducing the \emph{particle system higher-order fields} or, shortly, the \emph{higher-order fields}, in terms of joint factorial moments of the particle configurations. For this purpose, let  $\N_0^{\Z^d}=\left\{0,1,\ldots \right\}^{\Z^d}$ denote the set of configurations with $\eta(x)$ indicating the number of particles at site $x \in \Z^d$ for the configuration $\eta \in \N_0^{\Z^d}$. Then,   for all $k \in \N$, $N \in \N$ and $\eta \in \N_0^{\Z^d}$, the $k$th-order field $\mathscr X^{(k)N}$ associated with $\eta \in \N_0^{\Z^d}$ is given by
\begin{equation}\label{eq:fields}
\mathscr X^{(k)N}:= \frac{1}{N^{kd}}\sum_{\bx \in (\Z^d)^k} \delta_{\frac{\bx}{N}}\, [\eta]_\bx = \frac{1}{N^{kd}} \sum_{x_1}\cdots \sum_{x_k} \delta_{\frac{x_1}{N}}\otimes \cdots\otimes  \delta_{\frac{x_k}{N}}\,  [\eta]_{(x_1,\ldots, x_k)}\ ,
\end{equation}
where $\delta_{\frac{\bx}{N}}$ denotes the Dirac measure on $\frac{\bx}{N} \in \frac{(\Z^d)^k}{N}$, while $[\eta]_\bx \geq 0$ stands for the following joint falling factorial of $\eta = \left\{\eta(x): x \in \Z^d\right\} \in \N_0^{\Z^d}$:
\begin{equation}\label{eq:x-factorial_moment}
[\eta]_\bx=[\eta]_{(x_1,\ldots,x_k)} := \eta(x_1)\left(\eta(x_2)-\1_{\{x_2=x_1\}} \right)\cdots \left(\eta(x_k)-\sum_{j=1}^{k-1} \1_{\{x_k=x_j\}} \right)\ .	
\end{equation}
Above and in what follows, when the range of the summands is not indicated (as on the r.h.s.\ of \eqref{eq:fields}) it is understood that the summations run	 over $\Z^d$; an analogous convention will hold for suprema (see e.g.\ \eqref{eq:fN1}--\eqref{eq:fN2} below). We note that, for the particular choice $k=1$, the field in \eqref{eq:fields} corresponds to the most standard \emph{empirical density field} (\cite{de_masi_mathematical_1991}) for the configuration $\eta \in \N_0^{\Z^d}$:
\begin{equation}\label{eq:first_order_field}
\mathscr X^{(1)N}:= \frac{1}{N^d} \sum_x \delta_{\frac{x}{N}}\, \eta(x)\ .
\end{equation}
\begin{remark}[\textsc{an equivalent definition}]
The higher-order field $\mathscr X^{(k)N}$ in \eqref{eq:fields}  also arises as  empirical measure of distinct $k$-tuples of particles from the configuration $\eta \in \N_0^{\Z^d}$ as in e.g.\ \cite[Remark 2.3.1]{sznitman1991topics}. More specifically, let us consider a (possibly infinite) configuration $\eta \in \N_0^{\Z^d}$ and let $\bx = (x_1,x_2,\ldots) \in (\Z^d)^{\mathcal I}$ denote a configuration of labeled particles (with labels $i \in \mathcal I\subseteq \N$) \textquotedblleft compatible\textquotedblright\ with $\eta \in \N_0^{\Z^d}$, i.e.
$	\eta(x)= \sum_{i \in \mathcal I} \1_{\{x_i=x\}}
	$
	for all $x \in \Z^d$. Then, 
	\begin{equation}\label{eq:fields_2ndversion}
	\mathscr X^{(k)N}= \frac{1}{N^{kd}} \sum_{\substack{\{i_1,\ldots, i_k\} \subseteq \mathcal I\\
	i_1,\ldots,i_k\ \text{\normalfont distinct}}} \delta_{\frac{x_{i_1}}{N}}\otimes \cdots \otimes  \delta_{\frac{x_{i_k}}{N}}\ .	
	\end{equation}
	\end{remark}

To rigorously define higher-order fields, let, for all $k \in \N$,  $\mathscr S^{(k)}:= \otimes_{i=1}^k \mathscr S(\R^d) = \mathscr S(\R^{kd})$ be the $k$-fold tensor product of the Schwartz space $\mathscr S(\R^d)$ of smooth and rapidly decreasing functions on $\R^d$, with $(\mathscr S^{(k)})'$ denoting its strong topological dual. We will employ the notation $\langle \cdot, \cdot \rangle$ for the dual pairing between elements in $\mathscr S^{(k)}$ and elements in $(\mathscr S^{(k)})'$: for all $G \in \mathscr S^{(k)}$ and $\mathscr X \in (\mathscr S^{(k)})'$, 
\begin{equation*}\langle G, \mathscr X\rangle= \mathscr X(G)\ .
\end{equation*} 

A sufficient condition to ensure that the higher-order fields take values in the space of tempered distributions is to restrict the set of particle configurations.
To this aim, we introduce the  set of   configurations growing at infinity at most polynomially, i.e.\ 
\begin{equation}\label{eq:configurations_polynomial_growth}
	 \bigcup_{m, n \in \N} \mathcal X_{m,n}:= \bigcup_{m,n \in \N}\left\{\eta \in \N_0^{\Z^d}: \eta(x)\leq m\left(1+|x|\right)^n\ \text{for all}\ x \in \Z^d \right\}\ .
\end{equation}
Then, the $k$th-order field $\mathscr X^{(k)N}$ associated with any of such configurations 	is clearly an element of $(\mathscr S^{(k)})'$.

Let us further observe that $[\eta]_\bx$ is invariant under permutation of the indices of $\bx=(x_1,\ldots, x_k) \in (\Z^d)^k$, i.e., for all $\varsigma \in \Sigma_k$ and $\bx \in (\Z^d)^k$,
\begin{equation}\label{eq:permutation_invariance}
[\eta]_{\varsigma \bx} = [\eta]_\bx\ .
\end{equation}
As a  consequence, for all $G \in \mathscr S^{(k)}$, we have
\begin{align*}
\langle G,\mathscr X^{(k)N}\rangle = \langle G^\sy, \mathscr X^{(k)N}\rangle\ ,
\end{align*}
where
\begin{equation}
G^\sy := \frac{1}{k!} \sum_{\varsigma \in \Sigma_k} G\circ\varsigma\ .
\end{equation}
Moreover, all definitions above trivially extend to the case $k =0$ by setting 
\begin{equation*}
(\Z^d)^0:= \{\emptyset\}\  ,\qquad \mathscr S^{(0)}:= \{G: \{\emptyset\}\to \R \}\ ,
\end{equation*}
and $\mathscr X^{(0)N}$ such that
\begin{align*}
\langle G, \mathscr X^{(0)N}\rangle= G(\emptyset)
\end{align*}
holds for all $\eta \in \N_0^{\Z^d}$,  $G \in \mathscr S^{(0)}$ and $N \in \N$.

\subsubsection{Products  of higher-order fields} 
The following basic formula provides an expansion of products of projections of higher-order fields or, shortly,  \textquotedblleft  products of higher-order fields\textquotedblright. This latter terminology is imprecise. Indeed,  we do not  consider products of elements $\mathscr X^{(k)N}$ and $\mathscr X^{(\ell)N}$ in $(\mathscr S^{(k)})'$ and $(\mathscr S^{(\ell)})'$, respectively, but rather  products between their	 projections $\langle G, \mathscr X^{(k)N}\rangle$ and $\langle H, \mathscr X^{(\ell)N}\rangle$,  for some $G \in \mathscr S^{(k)}$ and $H \in \mathscr S^{(\ell)}$.  In the remainder of this section,  $\eta \in \bigcup_{m,n\in \N} \mathcal X_{m,n} \subseteq \N_0^{\Z^d}$ and the associated higher-order fields are fixed. 	
\begin{lemma}[\textsc{products of higher-order fields}]\label{lemma:product_fields}
	Let $k, \ell \in \N$ with $\ell \leq k$ and $G \in \mathscr S^{(k)}$, $H \in \mathscr S^{(\ell)}$. Then, for all $N \in \N$, we have
	\begin{equation}
	\langle G, \mathscr X^{(k)N)}\rangle\, \langle H,\mathscr X^{(\ell)N}\rangle = \sum_{h=0}^\ell \frac{1}{N^{hd}} \langle \left\{G \otimes H\right\}^{(k+\ell-h)}, \mathscr X^{(k+\ell-h)N}\rangle\ ,\\
	\end{equation}
	where, for all $h \in [\ell]_0$,  $\left\{G \otimes H\right\}^{(k+\ell-h)} \in \mathscr S^{(k+\ell-h)}$   satisfies
	\begin{equation}\label{eq:summation_klh}
	\langle \left\{G\otimes H \right\}^{(k+\ell-h)},\mathscr X^{(k+\ell-h)N}\rangle=\frac{1}{N^{kd+\ell d-hd}}\sum_{\bx \in (\Z^d)^k}\sum_{\by \in (\Z^d)^\ell} G(\tfrac{\bx}{N}) H(\tfrac{\by}{N}) \left\{\eta|(\bx,\by)\right\}^{(k+\ell-h)}\ , 
	\end{equation}
	with, for given  $\bx \in (\Z^d)^k$ and $\by \in (\Z^d)^\ell$, 
	\begin{equation}\label{eq:definition_strange_etaxy}
	\left\{\eta|(\bx,\by) \right\}^{(k+\ell-h)}:= \sum_{\substack{\mathcal J\subseteq [\ell]\\
	|\mathcal J|=h}} [\eta]_{\bx\:\widehat \by_{{\mathcal J}}} \sum_{\substack{i : \mathcal J\to [k]\\\text{\normalfont one-to-one}}} \prod_{j \in \mathcal J}\1_{\{y_j = x_{i_j}\}} \ .
	\end{equation}
	\end{lemma}
\begin{proof}
	In what follows,  anytime we have a denominator, the corresponding summation is meant to run only over the sites of $\Z^d$ for which the denominator is non-zero. Hence, 	
	\begin{align}\label{eq:product_fields}
	\nonumber
	&\langle G, \mathscr X^{(k)N)}\rangle\, \langle H,\mathscr X^{(\ell)N}\rangle\\
	\nonumber
	=&\ 	\frac{1}{N^{kd+\ell d}} \sum_{\bx \in (\Z^d)^k} \sum_{\by \in (\Z^d)^\ell} G(\tfrac{\bx}{N})\, H(\tfrac{\by}{N})\, [\eta]_\bx\, [\eta]_\by\\
	\nonumber
	=&\ \frac{1}{N^{kd+\ell d}} \sum_{x_1}\cdots \sum_{x_k} G(\tfrac{x_1}{N},\ldots, \tfrac{x_k}{N}) \sum_{y_1} [\eta]_{(x_1,\ldots,x_k)}\, \eta(y_1) \sum_{y_2}\cdots \sum_{y_\ell} H(\tfrac{y_1}{N},\ldots,\tfrac{y_\ell}{N})\, \frac{[\eta]_{(y_1,\ldots,y_\ell)}}{\eta(y_1)}\\
	\nonumber
	=&\ \frac{1}{N^{k d+\ell d}} \sum_{x_1}\cdots \sum_{x_k} \sum_{y_1} G(\tfrac{x_1}{N},\ldots, \tfrac{x_k}{N})\, [\eta]_{(x_1,\ldots,x_k,y_1)} \sum_{y_2}\cdots \sum_{y_\ell} H(\tfrac{y_1}{N},\ldots,\tfrac{y_\ell}{N})\, \frac{[\eta]_{(y_1,\ldots,y_\ell)}}{\eta(y_1)}\\
	+&\ \sum_{i_1=1}^k \frac{1}{N^{kd+\ell d}} \sum_{x_1} \cdots \sum_{x_k}  G(\tfrac{x_1}{N},\ldots, \tfrac{x_k}{N})\, [\eta]_{(x_1,\ldots,x_k)} \sum_{y_2} \cdots \sum_{y_\ell} H(\tfrac{x_{i_1}}{N},\tfrac{y_2}{N},\ldots,\tfrac{y_\ell}{N}) \frac{[\eta]_{(x_{i_1},y_2,\ldots,y_\ell)}}{\eta(x_{i_1})}\ ,
	\end{align}
	where the last identity is a consequence of
	\begin{align*}
	\eta(y_1) = \left(\eta(y_1) - \sum_{i_1=1}^k \1_{\{y_1=x_{i_1}\}} \right) + \sum_{i_1=1}^k \1_{\{y_1=x_{i_1}\}}
	\end{align*}
	and the definition  of $[\eta]_\bx$ in \eqref{eq:x-factorial_moment}. As a consequence of the following two identities
	\begin{align*}
	[\eta]_{(x_1,\ldots,x_k,y_1)}\,\frac{[\eta]_{(y_1,\ldots,y_\ell)}}{\eta(y_1)} =&\ [\eta]_{(x_1,\ldots,x_k,y_1,y_2)}\, \frac{[\eta]_{(y_1,\ldots,y_\ell)}}{\eta(y_1)\left(\eta(y_2)-\1_{\{y_2=y_1\}} \right)}\\
	+&\ \sum_{i_2=1}^k [\eta]_{(x_1,\ldots,x_k,y_1)}\, \frac{[\eta]_{(y_1,x_{i_2},y_3,\ldots,y_\ell)}}{\eta(y_1) \left(\eta(x_{i_2})-\1_{\{x_{i_2}=y_1\}} \right)}
	\end{align*}
	and, for all $i_1 \in \{1,\ldots,k\}$,
	\begin{align*}
	[\eta]_{(x_1,\ldots,x_k)}\, \frac{[\eta]_{(x_{i_1},y_2,\ldots,y_\ell)}}{\eta(y_{x_{i_1}})} =&\ [\eta]_{(x_1,\ldots,x_k,y_2)}\,  \frac{[\eta]_{(x_{i_1},y_2,\ldots,y_\ell)}}{\eta(y_{x_{i_1}})\left(\eta(y_2)-\1_{\{y_2=x_{i_1}\}} \right)}\\
	+&\ \sum_{\substack{i_2=1\\i_2 \neq i_1}}^k [\eta]_{(x_1,\ldots,x_k)}\, \frac{[\eta]_{(x_{i_1},x_{i_2},y_3,\ldots,y_\ell)}}{\eta(x_{i_1})\left(\eta(x_{i_2})-\1_{\{x_{i_2}=x_{i_1}\}} \right)}\ ,
	\end{align*}
	we further write the expression in \eqref{eq:product_fields} as follows:
	\begin{align*}
	&\langle G, \mathscr X^{(k)N)}\rangle\, \langle H,\mathscr X^{(\ell)N}\rangle\\
	=&\ \frac{1}{N^{k d+\ell d}} \sum_{x_1}\cdots \sum_{x_k}\sum_{y_1}\sum_{y_2} G(\tfrac{x_1}{N},\ldots,\tfrac{x_k}{N})\, [\eta]_{(x_1,\ldots,x_k,y_1,y_2)}\\
	&\qquad \times \sum_{y_3}\cdots \sum_{y_\ell} H(\tfrac{y_1}{N},\ldots,\tfrac{y_\ell}{N})\, \frac{[\eta]_{(y_1,\ldots,y_\ell)}}{\eta(y_1)\left(\eta(y_2)-\1_{\{y_2=y_2\}} \right)}\\
	+&\ \frac{1}{N^d} \sum_{i_1=1}^k \frac{1}{N^{k d+\ell d-d}} \sum_{x_1}\cdots \sum_{x_k} \sum_{y_2} G(\tfrac{x_1}{N},\ldots,\tfrac{x_k}{N})\, [\eta]_{(x_1,\ldots,x_k,y_2)}\\
	&\qquad\times \sum_{y_3}\cdots \sum_{y_\ell} H(\tfrac{x_{i_1}}{N},\tfrac{y_2}{N},\ldots,\tfrac{y_\ell}{N})\, \frac{[\eta]_{(x_{i_1},y_2,\ldots,y_\ell)}}{\eta(y_{x_{i_1}})\left(\eta(y_2)-\1_{\{y_2=x_{i_1}\}} \right)}\\
	+&\ \frac{1}{N^d}\sum_{i_2=1}^k \frac{1}{N^{k d+\ell d-d}} \sum_{x_1}\cdots \sum_{x_k}\sum_{y_1} G(\tfrac{x_1}{N},\ldots,\tfrac{x_k}{N})\, [\eta]_{(x_1,\ldots,x_k,y_1)}\\
	&\qquad \times \sum_{y_3}\cdots \sum_{y_\ell} H(\tfrac{y_1}{N},\tfrac{x_{i_2}}{N},\tfrac{y_3}{N},\ldots,\tfrac{y_\ell}{N})\, \frac{[\eta]_{(y_1,x_{i_2},y_3,\ldots,y_\ell)}}{\eta(y_1) \left(\eta(x_{i_2})-\1_{\{x_{i_2}=y_1\}} \right)}\\
	+&\ \frac{1}{N^{2d}}\sum_{i_1=1}^k \sum_{\substack{i_2=1\\i_2\neq i_1}}^k \frac{1}{N^{kd+\ell d-2d}} \sum_{x_1}\cdots \sum_{x_k} G(\tfrac{x_1}{N},\ldots,\tfrac{x_k}{N})\, [\eta]_{(x_1,\ldots,x_k)} \\
	&\qquad \times\sum_{y_3}\cdots \sum_{y_\ell} H(\tfrac{x_{i_1}}{N},\tfrac{x_{i_2}}{N},\tfrac{y_3}{N},\ldots,\tfrac{y_\ell}{N})\, \frac{[\eta]_{(x_{i_1},x_{i_2},y_3,\ldots,y_\ell)}}{\eta(x_{i_1})\left(\eta(x_{i_2})-\1_{\{x_{i_2}=x_{i_1}\}} \right)}\ .
	\end{align*} 
	We note that, by the same arguments used so far, each of the four terms in the r.h.s.\ above may be further split into two terms: one which keeps the same number of sums over $\Z^d$ and another one consisting of an additional sum over $i_3 \in \{1,\ldots,k\}$ which \textquotedblleft replaces\textquotedblright\ the sum	 $\sum_{y_3}$. 	By iterating for a finite number of steps	 this procedure to all such terms,  we get the final result.
\end{proof}
We remark that  $\{G\otimes H\}^{(k+\ell-h)}$ may be explicitly recovered by rearranging the sums in \eqref{eq:summation_klh} and using the definition in \eqref{eq:definition_strange_etaxy}; however, we will not need this explicit expression in what follows. Moreover, $\{G \otimes H\}^{(k+\ell-h)} \in \mathscr S^{(k+\ell-h)}$ because all Schwartz spaces are closed under pointwise multiplication.	

We end this section by stating some properties of   the functions $\left\{G \otimes H \right\}^{(k+\ell-h)}$ and $\left\{\eta|(\bx,\by)\right\}^{(k+\ell-h)}$, which will be invoked in the proof of Theorem \ref{theorem:hydrodynamics} below. We omit their proofs as they follow at once from the definitions given in the statement of  \cref{lemma:product_fields} and the permutation invariance   \eqref{eq:permutation_invariance}.

\begin{proposition}\label{proposition:properties_product}
	Let us keep the same notation as in the statement of  \cref{lemma:product_fields}. Then, 
\begin{enumerate}[label={\normalfont (\arabic*)},ref={\normalfont(\arabic*)}]
\item \label{it:tensor}	For $h = 0$, the function $\left\{G \otimes H \right\}^{(k+\ell-h)}$ coincides with the usual tensor product $\otimes: \mathscr S^{(k)}\times \mathscr S^{(\ell)}\to \mathscr S^{(k+\ell)}$, i.e., for all $\bu=(u_1,\ldots, u_k,u_{k+1},\ldots, u_{k+\ell}) \in (\R^d)^{k+\ell}$, 	
\begin{align*}
\left\{G \otimes H \right\}^{(k+\ell)}(u_1,\ldots, u_k,u_{k+1},\ldots, u_{k+\ell}) =&\ G(u_1,\ldots, u_k) H(u_{k+1},\ldots, u_{k+\ell})\\=&\  \left(G\otimes H \right)(u_1,\ldots, u_k,u_{k+1},\ldots, u_{k+\ell})\ .
\end{align*}
Analogously, 
\begin{equation*}
\{\eta|(\bx,\by)\}^{(k+\ell)}=[\eta]_{\bx:\by}\ .	
\end{equation*}
\item \label{it:permutation2} For all $h \in [\ell]_0$ and for all $\bx \in (\Z^d)^k$, the following function
\begin{equation*}
g_\bx(\by):=\left\{\eta|(\bx,\by) \right\}^{(k+\ell-h)}
\end{equation*}
is invariant under permutation of the indices of $\by \in (\Z^d)^\ell$, i.e., for all $\varsigma \in \Sigma_\ell$ and $\by \in (\Z^d)^\ell$, 
\begin{equation*}
g_\bx(\by)=g_\bx(\varsigma \by)\ .
\end{equation*}
\item \label{it:permutation1}  For all $h \in [\ell]_0$,  the following function
\begin{align*}
f(\bx):= \frac{1}{N^{kd-\ell d}}\sum_{\by \in (\Z^d)^\ell} H(\tfrac{\by}{N}) \left\{\eta|(\bx,\by) \right\}^{(k+\ell-h)}
\end{align*}
is invariant under permutation of the indices of $\bx   \in (\Z^d)^k$, i.e., for all $\varsigma \in \Sigma_k$ and $\bx \in (\Z^d)^k$, 
\begin{equation*}
f(\bx)=f(\varsigma \bx)\ .
\end{equation*}
\end{enumerate}
\end{proposition}

\

\subsection{Interacting particle systems}\label{section:infinite_particle_systems}
Let us introduce three \textquotedblleft linear\textquotedblright\ interacting particle systems---the \emph{symmetric exclusion process} ($\SEP$), a system of \emph{independent random walkers} ($\IRW$), and the \emph{symmetric inclusion process} ($\SIP$)---for which we study scaling limits of the associated higher-order fields. \textquotedblleft Linearity\textquotedblright\ for these infinite particle systems corresponds to a notion of \textquotedblleft duality\textquotedblright\ which allows the description of the  evolution of suitable weighted factorial moments of the occupation variables in terms of a closed system of \emph{linear} evolution equations. Such a duality property will, in turn, yield linear lattice SPDEs (\cite{peszat_stochastic_2007}) for the corresponding higher-order fields.

For all $N \in \N$, the dynamics of the infinite particle systems we consider is described by the operator $\mathcal L^N$, whose action on local functions $f: \N_0^{\Z^d}\to \R$ reads as follows:
\begin{align}\label{eq:generator_infinite}
\mathcal L^N f(\eta):=  \sum_{x} \sum_{y} \1_{\{|x-y|=1\}}  N^2\left\{\begin{array}{r}\eta(x)\left(\alpha+\sigma \eta(y) \right)  \left(f(\eta^{x,y})-f(\eta) \right)\\[.15cm]
+\, \eta(y)\left(\alpha+\sigma \eta(x) \right)\left(f(\eta^{y,x})-f(\eta) \right)\end{array}\right\}\ ,
\end{align}
where   $\sigma \in \{-1,0,1\}$ and
\begin{align*}
\eta^{x,y}(z):=\begin{dcases}\eta(x)-1 &\text{if}\ z=x\ \text{and}\ \eta(x)>0\\
\eta(y)+1 &\text{if}\ z=y\\
\eta(z) &\text{otherwise}
\end{dcases}
\end{align*}
indicates the configuration obtained from $\eta \in \N_0^{\Z^d}$ by removing a particle from $x \in \Z^d$ (if any) and placing it at $y \in \Z^d$. We incorporate a factor $N^2$ on the r.h.s.\ in \eqref{eq:generator_infinite}, which, together with the space rescaling by $N^{-1}$ included in the definition of higher-order fields, yields a diffusive space-time rescaling of the microscopic particle system. We further note that for local functions $f: \N_0^{\Z^d}\to \R$, i.e.\ functions which depend on $\eta \in \N_0^{\Z^d}$ only through finitely many variables $\{\eta(x): x \in \Z^d\}$, the  r.h.s.\ in \eqref{eq:generator_infinite} reduces to a finite summation.
Moreover, we assume $\alpha \in \N$, while the  parameter $\sigma \in \{-1,0,1\}$ corresponds to three different types of particle interaction, namely  \emph{exclusion} ($\sigma = -1$), \emph{inclusion} ($\sigma =1$) and \emph{no} interaction ($\sigma = 0$). For notational convenience, we shall suppress the symbol $\sigma$ in what follows.

In order to ensure non-negative rates and non-explosiveness of the infinite particle systems, we  need to restrict the set of configurations $\eta \in \N_0^{\Z^d}$ from which we start the dynamics. We refer, for each choice of $\sigma \in \{-1,0,1\}$, to $\mathcal  X$ as the subset of \textquotedblleft admissible\textquotedblright\ particle configurations. When $\sigma = -1$, the exclusion  dynamics must be clearly restricted to the subset of configurations with at most $\alpha \in \N$ particles per site. For $\sigma \in \{0,1\}$, we show in \cref{appendix:infinite_particle_systems} below that the infinite particle system Markovian dynamics is well-defined for all times and is fully supported on the subset of configurations growing at most polynomially defined in \eqref{eq:configurations_polynomial_growth}. More precisely, if	
\begin{equation}\label{eq:set_admissible}
\mathcal X= \begin{cases}
\mathcal X_{\alpha,0} &\text{if}\ \sigma = -1\\
\bigcup_{m, n \in \N} \mathcal X_{m,n} &\text{if}\ \sigma \in \{0,1\}\ ,
\end{cases}	
\end{equation}
then the operator $\mathcal L^N$ given in \eqref{eq:generator_infinite} generates a  Markovian dynamics on the state space $\mathcal X$.

For all probability measures $\mu$ on $\N_0^{\Z^d}$ and local functions $f: \mathcal X\to \R$,  let $E_\mu\left[f(\eta)\right]$ denote the expectation of $f$ w.r.t.\ $\mu$.
We recall from e.g.\ \cite{carinci_dualities_2015} that, for each choice of $\sigma \in \{-1,0,1\}$, the infinite particle system with generator $\mathcal L^N$ admits a one-parameter family of reversible product measures
\begin{equation}\label{eq:reversible_product_measures}
\left\{\mu_\scale: \scale \in \Scale \right\} = \left\{\otimes_{x \in \Z^d}\, \nu_{x,\scale}: \scale \in \Scale \right\}
\end{equation}
with 
\begin{equation}\label{eq:site-marginal_measures}
\Scale:=\begin{dcases}
[0,1] &\text{if}\ \sigma = -1\\
[0,\infty) &\text{if}\  \sigma \in \{0,1\}\ ,
\end{dcases}
\qquad \text{and}\qquad
\nu_{x,\scale}:= \begin{dcases}
\bin(\alpha,\scale) &\text{if}\ \sigma = -1\\
\poisson(\alpha \scale) &\text{if}\ \sigma = 0\\
\negbin(\alpha,\scale) &\text{if}\ \sigma = 1\ ,
\end{dcases}
\end{equation}
where the parametrization above is such that, for all $\sigma \in \{-1,0,1\}$, $\scale \in \Scale$ and $x \in \Z^d$, 
\begin{align*}
E_{\mu_\scale}\left[\eta(x) \right] = \alpha \scale\quad \text{and}\quad E_{\mu_\scale}\left[\left(\eta(x)-\alpha \scale \right)^2\right]=\alpha \scale(1+\sigma \scale )\ .
\end{align*}
We remark that  $\mu_\scale$ is  fully supported on $\mathcal X$ for all $\scale \in \Scale$. Indeed, while this is clearly the case for $\sigma = -1$, for the case $\sigma \in \{0,1\}$, a standard Borel-Cantelli argument (cf.\ e.g.\  \cite[p.\ 15]{de_masi_mathematical_1991}) shows, more generally, that   any probability measure $\mu$ on $\N_0^{\Z^d}$ such that
\begin{align*}
\sup_{x \in \Z^d} E_\mu\left[\eta(x) \right]<\infty\ ,
\end{align*}
is fully supported on the subset of configurations growing at infinity at most polynomially.

\subsubsection{Dual processes and duality relations}\label{section:duality} In this section, we introduce the dual processes and duality relations for the infinite particle systems with generator $\mathcal L^N$ defined in \eqref{eq:generator_infinite} above. In particular, for all choices of $\sigma \in \{-1,0,1\}$, the dual process consists in a finite  particle system, in which the particles	 undergo the same interaction rules as their infinite analogues. More precisely, let $k \in \N$ and  consider the following operator $A^{(k)N}$ given, for all  functions $f: (\Z^d)^k\to \R$, as
\begin{equation}\label{eq:generatorA}
A^{(k)N}f(\bx):= \sum_{i=1}^k A^{(k)N}_i f(\bx) + \sigma \sum_{i=1}^k \sum_{j=1}^k B^{(k)N}_{i,j}f(\bx)\ ,
\end{equation}
where, for all $i, j \in \{1,\ldots, k\}$,
\begin{equation}\label{eq:generatorB}
A^{(k)N}_i f(\bx):=  \sum_{y_i} \1_{\{|y_i-x_i|=1\}} N^2 \alpha \left(f(\bx_i^{y_i})-f(\bx_i^{x_i}) \right)
\end{equation}
and
\begin{equation}\label{eq:generatorC}
 B^{(k)N}_{i,j}f(\bx) = \1_{\{|x_i-x_j|=1\}}N^2 \left(f(\bx_i^{x_j})-f(\bx_i^{x_i}) \right)\ .
\end{equation}
Although the operator $ A^{(k)N}$ is not, in general, a Markov generator because for the case $\sigma =-1$ the rates are not necessarily non-negative and because we have not specified the function space on which the operator acts, formally, 	the dynamics described by  $A^{(k)N}$ consists of a non-interacting part, corresponding to the operators $A^{(k)N}_i$, and an interacting one only if $\sigma  \neq 0$, corresponding to the operators $ B^{(k)N}_{i,j}$. In particular, while no restriction on the set of labeled particle configurations is required when $\sigma \in \{0, 1\}$, for the case $\sigma = -1$ we discard from  $(\Z^d)^k$ the subset $(\Z^d)^k_\ast$ given by
\begin{align*}
(\Z^d)^k_\ast:= \left\{\bx \in (\Z^d)^k: \sup_{x \in \Z^d}\sum_{i=1}^k \1_{\{x_i=x\}} > \alpha \right\}\ .
\end{align*}
As a standard detailed balance (see e.g.\ \cite{kipnis_scaling_1999}) computation shows,   for all $\sigma \in \{-1,0,1\}$, the operator $A^{(k)N}$ is self-adjoint in $\ell^2_\pi((\Z^d)^k)$, where  $\pi$ denotes the  (infinite) measure on $(\Z^d)^k$ given, for all $\bx \in (\Z^d)^k$,  by
\begin{align}\label{eq:pi_measure}
\pi(\bx):= \alpha \left(\alpha+\sigma \1_{\{x_2=x_1\}} \right)\cdots \left(\alpha+\sigma \sum_{j=1}^{k-1}\1_{\{x_k=x_j\}} \right)\ .
\end{align}
In particular, we observe that $\pi(\bx)=0$ if both $\sigma = -1$ and $\bx \in  (\Z^d)^k_\ast$, while $\pi(\bx) > 0$ otherwise. Moreover, for all $k \in \N$,
\begin{equation}\label{eq:upper_bound_pi}
\cpi^{(k)}:=\sup_{\bx \in (\Z^d)^k} \pi(\bx) 	 <\infty\ .
\end{equation} As an immediate consequence, 
 for all $f, g : (\Z^d)^k\to \R$ for which both sides below are finite, we have
\begin{align}\label{eq:self-adjoint}
\sum_{\bx \in (\Z^d)^k} g(\bx)\, A^{(k)N}f(\bx)\, \pi(\bx) = \sum_{\bx \in (\Z^d)^k} A^{(k)N}g(\bx)\, f(\bx)\, \pi(\bx)\ ,
\end{align}
and $A^{(k)N}$ is a bounded Markov generator as an operator  in 
$\ell^\infty_\pi((\Z^d)^k)$, giving rise, for all $\sigma \in \{-1,0,1\}$, to a well-defined countable state space Markov process. We let $\{\mathbf X^{\bx,N}_t: t \geq 0 \}$, resp.\ $\widehat \Pr^N_\bx$ and $\widehat \E^N_\bx$, denote such Markov process on $(\Z^d)^k$ ($(\Z^d)^k \setminus (\Z^d)^k_\ast$ if $\sigma = -1$) with generator $A^{(k)N}$ when started from $\bx \in (\Z^d)^k$ ($(\Z^d)^k\setminus (\Z^d)^k_\ast$ if $\sigma = -1$), resp.\ its probability law and corresponding expectation.

\

We recall from e.g.\ \cite{giardina_duality_2009} (see also \cite[\S VIII.1]{liggett_interacting_2005-1} and \cite[\S 6.3]{de_masi_mathematical_1991} for $\SEP$ ($\sigma = -1$) and \cite[\S 2.9.2]{de_masi_mathematical_1991} for  $\IRW$ ($\sigma =0$)) that, for all $\sigma \in \{-1,0,1\}$ and for all $k \in \N$, the  processes associated with the Markov generators $\mathcal L^N$ and $A^{(k)N}$ are \emph{dual}. More precisely, let us define the following function $D: (\Z^d)^k\times \mathcal X \to \R$, for all $\bx \in (\Z^d)^k$ and $\eta \in \mathcal X$, as
\begin{align*}
D(\bx,\eta):=\frac{[\eta]_\bx}{\pi(\bx)}\ .
\end{align*}
Then, the function $D(\bx,\eta)$ is a \emph{duality function} for the  processes associated with the Markov generators $\mathcal L^N$ and $A^{(k)N}$, i.e.\ the following duality relation
\begin{align}\label{eq:duality_relation}
A^{(k)N}D(\cdot,\eta)(\bx)\, \pi(\bx) = \mathcal L^N D(\bx,\cdot)(\eta)\, \pi(\bx)
\end{align}
holds for all $\bx\in (\Z^d)^k$ and $\eta \in \mathcal X$. We note that $D(\bx,\cdot):\mathcal X\to \R$ is a local function   for all $\bx \in (\Z^d)^k$, hence, both sides in \eqref{eq:duality_relation} reduce to finite summations; moreover, if $\sigma = -1$ and $\bx \in (\Z^d)^k_\ast$, 	 both sides equal zero. Note	 that the unlabeled version of the particle system $\{\mathbf X^{\bx,N}_t: t \geq 0\}$ is Markov with generator $\mathcal L^N$. For this reason,  the  duality relation \eqref{eq:duality_relation} between the labeled and the unlabeled versions of the same	particle system represents an instance of \emph{self-duality}.

\section{Main results}\label{section:main_results}
\subsection{Higher-order hydrodynamic limit}\label{section:hydrodynamics}
In this section we present our first main result concerning the hydrodynamic limit for higher-order fields of linear interacting particle systems. For this purpose, let us first introduce some notation.

For all choices of  $\sigma \in \{-1,0,1\}$, for all  $N \in \N$ and for all probability measures $\mu^N$ on $\mathcal X$,  $\Pr^N_{\mu^N}$ and $\E^N_{\mu^N}$ indicate the probability law and the corresponding expectation of the Markov process with state space $\mathcal X$ and with generator $\mathcal L^N$ given in \eqref{eq:generator_infinite} (see also \cref{appendix:infinite_particle_systems}). Let  
\begin{equation}\label{eq:Markov_particle_system}
\left\{\eta^N_t: t \geq 0 \right\}
\end{equation}
denote such Markov process. Then,  we introduce, for all $k \in \N$, the following $(\mathscr S^{(k)})'$-valued stochastic process
\begin{equation}\label{eq:fields_process}
\left\{\mathscr X^{(k)N}_t: t \geq 0 \right\}\ ,
\end{equation}
given, for all $t \geq 0$ and $G \in \mathscr S^{(k)}$, by	
\begin{align*}
\langle G, \mathscr X^{(k)N}_t\rangle= \frac{1}{N^{kd}}\sum_{\bx \in (\Z^d)^k} G(\tfrac{\bx}{N})\, [\eta^N_t]_\bx\ .	
\end{align*}
For notational convenience, we will not distinguish between the probability laws and expectations of the two processes in \eqref{eq:Markov_particle_system} and \eqref{eq:fields_process}.

Before presenting the statement of the main theorem of this section,  we need the following definition.

\begin{definition}[\textsc{weak law of large numbers at the initial time}]\label{definition:associated_profile}
	Let $\sigma \in \{-1,0,1\}$,  $k \in \N$ and a sequence of probability measures $\left\{\mu^N: N \in \N \right\}$ on $\mathcal X$ be given. We say  that a law of large numbers  at the initial time for $\left\{\mu^N: N \in \N \right\}$ with profile $ \ssscale \in (\mathscr S^{(k)})'$ holds if, for all $\delta > 0$ and $G \in \mathscr S^{(k)}$, we have
	\begin{align*}
	\Pr^N_{\mu^N}\left(\left|\langle G, \mathscr X^{(k)N}_0\rangle - \langle G,\ssscale \rangle\right|>\delta\right)\underset{N\to \infty}\longrightarrow 0\ .
	\end{align*}
\end{definition}

\begin{theorem}[\textsc{hydrodynamic limit}]\label{theorem:hydrodynamics}
	Let $k \in \N$ be fixed and,  for all $\sigma \in \{-1,0,1\}$, let $\left\{\mu^N: N \in \N \right\}$ be a sequence of probability measures on $\mathcal X$ satisfying the following two assumptions:
	\begin{enumerate}[label={\normalfont (\alph*)},ref={\normalfont (\alph*)}]
		\item \label{it:assumption_LLN}	A weak law of large numbers at the initial time for $\left\{\mu^N: N \in \N \right\}$ with profile $ \ssscale \in (\mathscr S^{(k)})'$ holds.
		\smallskip	
	\item \label{it:assumption_bound} There exists a constant $\scale \in \Scale$ such that, for all $\ell \in \N$, $\bx \in (\Z^d)^\ell$ and $N \in \N$, we have
		\begin{equation*}
		 E_{\mu^N}[[\eta]_\bx]\leq  \scale^\ell \pi(\bx)\ .
		\end{equation*}

	\end{enumerate}
Then the following convergence in law in the Skorokhod space of tempered distribution-valued trajectories $\mathcal D([0,\infty),(\mathscr S^{(k)})')$ (see e.g.\ \cite{kallianpur_xiong_1995})
\begin{equation}\label{eq:convergence_in_law}
\left\{\mathscr X^{(k)N}_t: t \geq 0 \right\}\underset{N\to \infty}\Longrightarrow \left\{ \sscale_t: t \geq 0 \right\}
\end{equation}
holds, where $\left\{ \sscale_t: t \geq 0 \right\}$ is the unique deterministic solution in $\mathcal C([0,\infty),(\mathscr S^{(k)})')$, the space of continuous $(\mathscr S^{(k)})'$-valued trajectories (see.\ e.g.\ \cite{kallianpur_xiong_1995}), of the following identity:
\begin{align}\label{eq:deterministic_heat_equation}
 \langle G,  \sscale_t\rangle = \langle G,  \ssscale\rangle + \int_0^t \langle \mathscr A^{(k)}G,  \sscale_s\rangle\, \dd s\ ,\qquad t\geq 0\ ,\quad G \in \mathscr S^{(k)}\ ,
\end{align}
where $\mathscr A^{(k)}:\mathscr S^{(k)}\to \mathscr S^{(k)}$ is the  linear, bounded operator given by
\begin{equation}\label{eq:k-laplacian}
\mathscr A^{(k)}G:= \sum_{i=1}^k \mathscr A^{(k)}_i G:=\sum_{i=1}^k \tfrac{\alpha}{2} \Delta^{(k)}_i G\ ,
\end{equation}
with $\Delta^{(k)}_iG$ denoting the Laplacian of $G \in \mathscr S^{(k)}$ w.r.t.\ the $i$th coordinate in $(\R^d)^k = \R^d \times \cdots \times \R^d$.
	\end{theorem}
The proof of \cref{theorem:hydrodynamics} is given in \cref{section:proof_hydrodynamics} below.
Let us collect now some immediate observations on the assumptions \ref{it:assumption_LLN} and \ref{it:assumption_bound}  in  \cref{theorem:hydrodynamics}:
	\begin{enumerate}[label={\normalfont (\roman*)}]
		\item Assumption \ref{it:assumption_bound} is redundant for the case $\sigma = -1$ because of the a.s.\ uniform bound on the maximal number of particles per site. Indeed, the choice $\scale = 1$ would suffice. Nonetheless, for the sake of notational convenience, we decide to state this condition for all choices of $\sigma \in \{-1,0,1\}$.
		\smallskip
		\item For $\sigma \in \{0,1\}$,  assumption \ref{it:assumption_bound} implies, in particular, that all measures $\left\{\mu^N: N\in \N\right\}$ are fully supported on the subset of admissible configurations $\mathcal X$ (cf.\ \cref{section:infinite_particle_systems}).
		\smallskip
		\item   Let us recall the definitions \eqref{eq:reversible_product_measures} and \eqref{eq:site-marginal_measures}. Then, for all $\sigma \in \{-1,0,1\}$, the product measures with slowly varying profile (\cite{kipnis_scaling_1999})
		\begin{equation}\label{eq:local_Gibbs}
		\left\{\mu^N: N \in \N \right\}=\left\{\otimes_{x \in \Z^d}\, \nu_{x,\scale(\frac{x}{N})}: N \in \N \right\}	
		\end{equation}
		associated with the bounded function $\scale: \R^d \to \Scale$ satisfy assumption \ref{it:assumption_bound} in view of the following well-known formula on the factorial moments of Binomial, Poisson and Negative-Binomial distributions: for all $\ell \in \N$, $\bx \in (\Z^d)^\ell$ and $N \in \N$,  
		\begin{align*}
		E_{\mu^N}\left[[\eta]_{\bx}\right] = \left(\prod_{i=1}^k \scale(\tfrac{x_i}{N})\right)\pi(\bx) \ .
		\end{align*}
		 If, additionally, the function $\scale: \R^d\to \Scale$ is piecewise continuous, then \cref{theorem:hydrodynamics} holds: the corresponding product measures in \eqref{eq:local_Gibbs} also satisfy assumption \ref{it:assumption_LLN} with profile $  \ssscale(\dd u) = \otimes_{i=1}^k\, \alpha\, \scale(u)\, \dd u$ absolutely continuous w.r.t.\  the $(d\times k)$-dimensional Lebesgue measure and such that, for all $G = \otimes_{i=1}^k\, g_i\in \mathscr S^{(k)}$, 
		\begin{equation*}
		\langle G,\otimes_{i=1}^k\left(\alpha\, \scale(\cdot)\, \dd u\right)\rangle = \prod_{i=1}^k\left\{\int_{\R^d} g_i(u)\, \alpha \scale(u)\, \dd u \right\} \ .	
		\end{equation*}
		As a particular instance, taking $\scale: \R^d\to   \Scale$ constant, from \eqref{eq:local_Gibbs} we recover the reversible product measure $\mu_\scale$ and, as deterministic limit in \eqref{eq:convergence_in_law}, we obtain the stationary solution of \eqref{eq:deterministic_heat_equation} given, for all $G = \otimes_{i=1}^k\, g_i \in \mathscr S^{(k)}$, by
		\begin{align*}
	\langle G, \sscale_t\rangle = \prod_{i=1}^k \left\{\int_{\R^d} g_i(u)\, \alpha \scale\, \dd u \right\}
		\end{align*}
		for all $t \geq 0$.
		\smallskip
		\item By keeping the same notation as in \cref{section:duality}, because of duality \eqref{eq:duality_relation} and Tonelli's theorem, the upper bound in assumption \ref{it:assumption_bound} holds at any later time $t > 0$: for all $\sigma \in \{-1,0,1\}$, $\ell \in \N$ and $\bx \in (\Z^d)^\ell$,
		\begin{align}\label{eq:assumption_bound_timet}
		\E^N_{\mu^N}\left[[\eta^N_t]_\bx \right]=&\	\E^N_{\mu^N}\left[\frac{[\eta^N_t]_\bx}{\pi(\bx)} \right] \pi(\bx)	= \widehat \E^N_\bx\left[\frac{E_{\mu^N}\left[[\eta]_{\mathbf X^{\bx,N}_t}\right]}{\pi(\mathbf X^{\bx,N}_t)} \right] \pi(\bx)
		\leq  \scale^\ell \pi(\bx)\leq \scale^\ell \cpi^{(\ell)}\ .
		\end{align}
	\end{enumerate}

\subsection{Higher-order equilibrium fluctuations}\label{section:fluctuations}
Let us present our second main result concerning  equilibrium fluctuations for higher-order fields around their hydrodynamic limit. Let us  recall from \eqref{eq:reversible_product_measures} the definition of the reversible product measures $\{\mu_\scale: \scale \in \Scale\}$   for the interacting particle systems with generator $\mathcal L^N$ given in \eqref{eq:generator_infinite}. Then, for all  $k \in \N$, $N \in \N$ and $\scale \in \Scale$, we introduce 
\begin{equation}\label{eq:fluctuation_fields}
\left\{\mathscr Y^{(k,\scale)N}_t: t \geq 0 \right\} \in \mathcal D([0,\infty),(\mathscr S^{(k)})')
\end{equation}
as the field given in terms of  $\{\mathscr X^{(k)N}_t: t \geq 0\}$, for all $G \in \mathscr S^{(k)}$ and $t \geq 0$, by
\begin{equation}\label{eq:fluctuation_fields2}
\langle G, \mathscr Y^{(k,\scale)N}_t\rangle = N^{d/2} \left(\langle G, \mathscr X^{(k)N}_t\rangle- \E^N_{\mu_\scale}\left[\langle G, \mathscr X^{(k)N}_t\rangle \right] \right)\ .	
\end{equation}
We call these fields the \emph{$k$th-order fluctuation fields} associated with $\scale \in \Scale$.	Let us keep the same notation as in \cref{theorem:hydrodynamics} and observe that, for $\{\mu^N: N \in \N\}= \mu_\scale$ and for all $k \in \N$, \cref{theorem:hydrodynamics} applies, yielding a (weak) law of large numbers for the higher-order fields $\{\mathscr X^{(k)N}_\cdot: N \in \N \}$. In the following theorem, whose proof is postponed to  \cref{section:proof_fluctuations} below, we characterize the fluctuations of such fields by studying the limiting evolution of the stochastic processes given  in \eqref{eq:fluctuation_fields}--\eqref{eq:fluctuation_fields2}.

\begin{theorem}[\textsc{equilibrium fluctuations}]\label{theorem:fluctuations}
	For each choice of $\sigma \in \{-1,0,1\}$, for all $k \in \N$ and $\scale \in \Scale$, we have the following convergence in distribution in $\mathcal D([0,\infty),(\mathscr S^{(k)})')$	
	\begin{equation}\label{eq:convergence_in_law2}
	\left\{\mathscr Y^{(k,\scale)N}_t: t \geq 0 \right\}\underset{N\to \infty}\Longrightarrow \left\{\mathscr Y^{(k,\scale)}_t: t \geq 0 \right\}\ ,
	\end{equation} where $\{\mathscr Y^{(k,\scale)}_t: t \geq 0\} \in \mathcal C([0,\infty),(\mathscr S^{(k)})')$ denotes the  stationary Gaussian process  with distribution $\mathscr P^{(k,\scale)}$ and corresponding expectation $\mathscr E^{(k,\scale)}$, and is uniquely characterized by the following properties:
	\begin{itemize}
		\item The distribution of $\mathscr Y^{(k,\scale)}_0 \in (\mathscr S^{(k)})'$ is centered Gaussian with covariances given, for all product test functions $G= \otimes_{i=1}^k\, g_i$ and $H=\otimes_{i=1}^k\, h_i \in \mathscr S^{(k)}$, by
		\begin{align}
		\label{eq:covariance}
		&\ \mathscr E^{(k,\scale)}\left[\langle G, \mathscr Y^{(k,\scale)}_0\rangle \langle H, \mathscr Y^{(k,\scale)}_0\rangle \right]\\
			\nonumber
		=&\
		\sum_{i=1}^k \sum_{j=1}^k \left\{\int_{\R^d}g_i(u) h_j(u)\, \alpha \scale (1+\sigma \scale)\, \dd u \right\} \prod_{\substack{l=1\\l\neq i}}^k \left\{\int_{\R^d} g_l(u)\, \alpha\scale\, \dd u  \right\} \prod_{\substack{l'=1\\l'\neq j}}^k\left\{\int_{\R^d} h_{l'}(u)\, \alpha\scale\, \dd u  \right\}\ .
		\end{align}
		\item For all $G \in \mathscr S^{(k)}$, both continuous stochastic processes $\{\langle G, \mathscr M^{(k,\scale)}_t\rangle: t \geq 0\}$ and $\{\mathscr N^{(k,\scale)}_t(G): t \geq 0\}$ defined, respectively, for all $t \geq 0$, as 	
		\begin{equation}\label{eq:martingale_problem1}
		\langle G, \mathscr M^{(k,\scale)}_t\rangle:= \langle G, \mathscr Y^{(k,\scale)}_t\rangle -\langle G, \mathscr Y^{(k,\scale)}_0\rangle - \int_0^t \langle \mathscr A^{(k)}G, \mathscr Y^{(k,\scale)}_s\rangle\, \dd s
		\end{equation}
		and
		\begin{align}\label{eq:martingale_problem2}
		\mathscr N^{(k,\scale)}_t(G):=  \left(\langle G, \mathscr M^{(k,\scale)}_t\rangle \right)^2 -t\, \mathscr U^{(k,\scale)}(G)
		\end{align}
		are  integrable $\mathscr P^{(k,\scale)}$-martingales, where $\mathscr U^{(k,\scale)}(G)$ is 	deterministic, non-negative and given by
		\begin{equation}\label{eq:UK}
		\mathscr U^{(k,\scale)}(G)=\sum_{i=1}^k\sum_{j=1}^k \mathscr U^{(k,\scale)}_{\{i,j\}}(G)\  ,
		\end{equation}
		with, for $G = \otimes_{i=1}^k\, g_i \in \mathscr S^{(k)}$,
		\begin{align}\label{eq:UKij}
		\nonumber
		\mathscr U^{(k,\scale)}_{\{i,j\}}(G):=&\ \left\{\int_{\R^d}  \nabla g_i(u)\cdot  \nabla g_j(u)\, \alpha^2\scale (1+\sigma \scale)\, \dd u\right\}\\
		&\ \times
		 \prod_{\substack{l=1\\l\neq i}}^k \left\{\int_{\R^d} g_l(u)\, \alpha\scale\, \dd u  \right\} \prod_{\substack{l'=1\\l'\neq j}}^k\left\{\int_{\R^d} g_{l'}(u)\, \alpha\scale\, \dd u  \right\} \ .
		\end{align}
	\end{itemize} 	
\end{theorem} 

\begin{remark}[\textsc{covariance structure}]
	The form of the covariations in \eqref{eq:covariance} are reminiscent of that of inner products in the space of  tensor powers of $\mathscr S^{(1)}$. Indeed, for all $G = \otimes_{i=1}^k\, g_i \in \mathscr S^{(k)}$ and $H= \otimes_{i=1}^k\, h_i \in \mathscr S^{(k)}$, the structure of \eqref{eq:covariance} resembles that of the permanent of the matrix $((g_i,h_j))_{i,j=1}^k$ (see e.g. \cite[Exercise I.5.5]{bhatia2013matrix}), suitably rescaled by the mobility. We refer to \cite{minc1984permanents} for further references on permanents and their use in quantum field theory. 
	Let us further observe that an analogous structure appears in the predictable quadratic variation $\mathscr U^{(k,\scale)}(G)$ in \eqref{eq:UK}--\eqref{eq:UKij} with $\mathcal H_1$-inner products instead.
	
	In particular,  this connection with permanents and inner products in tensor product spaces explain not only the non-negativeness of $\mathscr U^{(k,\scale)}(G)$, for all $G \in \mathscr S^{(k)}$  and $k \in \N$, but also the tensorized structure of the stationary Gaussian process $\{\mathscr Y^{(k,\scale)}_t: t \geq 0\}$. Indeed, while the former is derived immediately from the Cauchy-Schwarz inequality (see e.g.\ \cite[\S 2.2]{minc1984permanents}), the latter follows from the tensorized structures of the Gaussian initial condition $\mathscr Y^{(k,\scale)}_0$,  the drift term on the r.h.s.\ of \eqref{eq:martingale_problem1} determined by $\mathscr A^{(k)}=\oplus_{i=1}^k\, \mathscr A^{(k)}_i$, and the white noise $\{\mathscr M^{(k,\scale)}_t: t \geq 0\}$ with predictable quadratic variations  \eqref{eq:UK}--\eqref{eq:UKij}.
\end{remark}

\subsubsection{Comparison with Theorem 5.1 in \cite{ayala2020higher}}\label{section:comparison}
	 \cref{theorem:fluctuations} above should be interpreted as a result determining the Gaussian limiting behavior of the equilibrium fluctuations of order $N^{d/2}$ for the higher-order fields $\{\mathscr X^{(k)N}_\cdot: N \in \N\}$ associated with the linear interacting particle systems introduced in \cref{section:infinite_particle_systems}.	A related result has been recently established in \cite{ayala2020higher}, for the same particle systems, although the equilibrium fluctuations considered there are of order $N^{kd/2}$.	Such dissimilarity may be explained in terms of different centering procedures used to define  the $k$th-order fluctuation fields. Indeed, let us introduce, for all $k, \ell \in \N_0$, $\ell \leq k$,  and $N \in \N$ the  operator $\mathscr K^{(k\:\ell)N}: \mathscr S^{(k)}\to \mathscr S^{(\ell)}$ given, for all $\bx \in (\Z^d)^\ell$, by
	 \begin{align*}
	 \mathscr K^{(k\:\ell)N}G(\tfrac{\bx}{N}):= \frac{1}{N^{kd-\ell d}}\sum_{\by \in (\Z^d)^{k-\ell}} G^\sy(\tfrac{\bx\:\by}{N}) \frac{\pi(\bx\:\by)}{\pi(\bx)} \ .
	 \end{align*}
	 Clearly, if $\ell=k$, then $\mathscr K^{(k\:k)}G=G^\sy$, while,  if $\ell =0$, we have
	 \begin{align*}
	 E_{\mu_\scale}\left[\langle G, \mathscr X^{(k)N}\rangle \right]= \scale^k \langle \mathscr K^{(k\:0)N}G, \mathscr X^{(0)N}\rangle\ ,	
	 \end{align*}
	 from which, by  \eqref{eq:fluctuation_fields2}, we get
	 \begin{equation}\label{eq:fluctuation_centering1}
	 \langle G, \mathscr Y^{(k,\scale)N}\rangle = N^{d/2} \left(\langle \mathscr K^{(k\:k)N}G, \mathscr X^{(k)N}\rangle - \scale^k \langle \mathscr K^{(k\:0)N}G, \mathscr X^{(0)N}\rangle\right)\ .	 \end{equation}
	 On the other side, the $k$th-order fluctuation fields $\{\mathscr Z^{(k,\scale)N}_t: t \geq 0\} \in \mathcal D([0,\infty), (\mathscr S^{(k)})')$ introduced in \cite[Eq.\ (30)]{ayala2020higher}  and defined in terms of orthogonal polynomial self-duality functions read as follows:
	 \begin{equation}\label{eq:fluctuation_centering2}
	 \langle G, \mathscr Z^{(k,\scale)N}\rangle= N^{kd/2}\sum_{\ell=0}^k \binom{k}{\ell} (-\scale)^{(k-\ell)} \langle \mathscr K^{(k\:\ell)N}G, \mathscr X^{(\ell)N}\rangle\ .	
	 \end{equation}
	 \cref{theorem:fluctuations} above and \cite[Theorem 5.1]{ayala2020higher} show that the different centering procedure employed in the two definitions  \eqref{eq:fluctuation_centering1} and \eqref{eq:fluctuation_centering2} is responsible for the difference in the size of the fluctuations (of order  $N^{-d/2}$ and $N^{-kd/2}$, respectively) and in the nature of the limiting predictable quadratic variations associated with the corresponding Dynkin's martingales (deterministic in \cref{theorem:fluctuations}, while stochastic and in terms of	 lower order fluctuation fields in \cite[Theorem 5.1]{ayala2020higher}). These findings suggest that several notions of equilibrium fluctuation fields---each detecting fluctuations of order $N^{-\ell d/2}$, $\ell \in [k]$, and corresponding to different centerings---are possible.
\section{Proof of \cref{theorem:hydrodynamics}}\label{section:proof_hydrodynamics}
Before turning to the details of the proof of  \cref{theorem:hydrodynamics}, we recall its main steps  below. 

First, 	by Dynkin's formula for Markov processes, for all $N \in \N$, $t \geq 0$ and  $G \in \mathscr S^{(k)}$, we have
\begin{equation}\label{eq:dynkin}
\langle G, \mathscr X^{(k)N}_t\rangle = \langle G, \mathscr X^{(k)N}_0\rangle + \int_0^t \mathcal L^N \langle G, \mathscr X^{(k)N}_s\rangle\, \dd s + \langle G, \mathscr M^{(k)N}_t\rangle\ ,
\end{equation}
with $\{\mathscr M^{(k)N}_t: t \geq 0 \}$ a $(\mathscr S^{(k)})'$-valued martingale with predictable quadratic covariations given, for all  $G, H \in \mathscr S^{(k)}$ and $t \geq 0$, by
\begin{align*}
\int_0^t\left\{\mathcal L^N\left(\langle G, \mathscr X^{(k)N}_s \rangle \langle H,\mathscr X^{(k)N}_s\rangle \right)  - \langle G, \mathscr X^{(k)N}_s\rangle\,  \mathcal L^N \langle H, \mathscr X^{(k)N}_s\rangle - \langle H, \mathscr X^{(k)N}_s\rangle\, \mathcal L^N \langle G, \mathscr X^{(k)N}_s\rangle \right\} \dd s\ .
\end{align*}
In view of the duality relation \eqref{eq:duality_relation}, for all $N \in \N$, $G \in \mathscr S^{(k)}$ and $t \geq 0$, we have
\begin{align*}
\mathcal L^N\langle G, \mathscr X^{(k)N}_t\rangle=&\ \frac{1}{N^{kd}}\sum_{\bx \in (\Z^d)^k} G(\tfrac{\bx}{N})\, \mathcal L^N [\eta^N_t]_\bx\\
=&\ \frac{1}{N^{kd}}\sum_{\bx \in (\Z^d)^k}G(\tfrac{\bx}{N})\, A^{(k)N}\left(\frac{[\eta^N_t]_\cdot}{\pi(\cdot)} \right)(\bx)\, \pi(\bx)= \langle \mathscr A^{(k)N}G,\mathscr X^{(k)N}_t\rangle\ ,
\end{align*}
where in the last identity we have used \eqref{eq:self-adjoint} with\footnote{If necessary, the definition of $\mathscr	 A^{(k)N}G$ may be extended  to  $(\R^d)^k \setminus ((\Z/N)^d)^k$  so to guarantee linearity and boundedness of $\mathscr A^{(k)N} : \mathscr S^{(k)}\to \mathscr S^{(k)}$.}
\begin{equation}\label{eq:generatorkN}
\mathscr A^{(k)N}G(\tfrac{\bx}{N}):= A^{(k)N}G(\tfrac{\cdot}{N})(\bx)\ .
\end{equation}
	
Being the limiting process $\{\sscale_t: t \geq 0\}$ deterministic, \cref{theorem:hydrodynamics} boils down  to show, for all $G \in \mathscr S^{(k)}$, $T > 0$ and $\delta > 0$,
\begin{align}\label{eq:convergence_probability}
&\ \Pr^N_{\mu^N}\left(\sup_{t \in [0,T]}\left|\langle G, \mathscr X^{(k)N}_t\rangle - \langle G, \sscale_t\rangle \right|> \delta \right)\underset{N\to \infty}\longrightarrow 0\ ,
\end{align}
which,  in view of both  decompositions in \eqref{eq:deterministic_heat_equation} and  \eqref{eq:dynkin}, is equivalent to
\begin{equation}\label{eq:convergence_probability2}
\Pr^N_{\mu^N}\left(\sup_{t \in [0,T]}\left|\left(\langle G, \mathscr X^{(k)N}_0\rangle-\langle G, \ssscale\rangle \right) + \int_0^t \langle \mathscr A^{(k)N}G-\mathscr A^{(k)}G, \mathscr X^{(k)N}_s\rangle\, \dd s + \langle G, \mathscr M^{(k)N}_t\rangle \right|> \delta \right)\underset{N\to \infty}\longrightarrow 0\ .
\end{equation}
For the purpose of proving \eqref{eq:convergence_probability}, at first we show it for  test functions $G \in \mathscr S^{(k)}$  in product form, i.e.\ 
\begin{equation*}G = \otimes_{i=1}^k\, g_i\quad \text{with}\quad \left\{g_i: i =1,\ldots, k\right\} \subseteq \mathscr S^{(1)}\ .
\end{equation*}
We remark that the fact that 	 $\mathscr A^{(k)N}_i$ and $\mathscr A^{(k)}_i$, $i \in [k]$, map product functions into product functions will be used repeatedly all throughout our proofs. As soon as  \eqref{eq:convergence_probability} holds for test functions,  the existence of a pure tensor product orthonormal basis in $L^2((\R^d)^k)$ spanning $\mathscr S^{(k)}$, the nuclear structure of $\mathscr S^{(k)}$ and a density argument complete the proof of \cref{theorem:hydrodynamics}; we postpone these last details to  Appendix \ref{appendix:test_functions}.

The proof of \eqref{eq:convergence_probability} for product test functions reduces to the proof that  each of the three terms between absolute value in \eqref{eq:convergence_probability2} vanish  as $N \to \infty$ in a suitable sense. More specifically, by means of Chebyshev's and Cauchy-Schwarz inequalities,
\begin{equation}\label{eq:convergence_integral_terms}
\sup_{t \geq 0} \E^N_{\mu^N}\left[\left(\langle \mathscr A^{(k)N}G - \mathscr A^{(k)}G, \mathscr X^{(k)N}_t\rangle \right)^2 \right]\underset{N\to \infty}\longrightarrow 0	
\end{equation}
takes care  of the term involving the time integral, while Chebyshev's and  Doob's martingale inequalities combined with
\begin{equation}\label{eq:vanishing_martingale}
\sup_{t \geq 0} \E^N_{\mu^N}\left[\left(\langle G, \mathscr M^{(k)N}_t\rangle \right)^2 \right]\underset{N\to \infty}\longrightarrow 0
\end{equation}
show that the martingale term in \eqref{eq:convergence_probability2} vanishes in the limit.
 The proofs of \eqref{eq:convergence_integral_terms} and \eqref{eq:vanishing_martingale} are the contents of, respectively,  Sections \ref{section:proof_integral_term} and \ref{section:carre_du_champ} below. For all functions $G \in \mathscr S^{(k)}$ for which \eqref{eq:convergence_integral_terms} and \eqref{eq:vanishing_martingale} hold, assumption \ref{it:assumption_LLN} of  \cref{theorem:hydrodynamics},  ensuring convergence at the initial time of the $k$th-order fields,  yields \eqref{eq:convergence_probability}; in particular,    assumption \ref{it:assumption_LLN} does not  play any role in the proof of both \eqref{eq:convergence_integral_terms} and \eqref{eq:vanishing_martingale}.
	
The proof of  \cref{theorem:hydrodynamics} ends by observing  that the limit points are supported on $\mathcal C([0,\infty),(\mathscr S^{(k)})')$ (the argument is standard since, due to Markovianity, particles jump one at the time, see e.g.\ \cite[\S 2.7]{de_masi_mathematical_1991}) and that the deterministic solution of the degenerate martingale problem \eqref{eq:deterministic_heat_equation} is unique in this space  (see e.g.\ \cite{holley_generalized_1978} with $B\equiv 0$ in Theorem 1.4).

\subsection{Proof of \eqref{eq:convergence_integral_terms}}\label{section:proof_integral_term}
In what follows, we fix $k \in \N$ and prove that, for all $G = \otimes_{i=1}^k\, g_i\in \mathscr S^{(k)}$ in product form,  \eqref{eq:convergence_integral_terms} holds true.
	For this purpose, let us recall \eqref{eq:generatorA}--\eqref{eq:generatorC} and define
	\begin{equation}\label{eq:definitionsBCN}
	\mathscr A^{(k)N}_iG(\tfrac{\bx}{N}):=\tfrac{\alpha}{2}\Delta^{(k)N}_iG(\tfrac{\bx}{N}):=  A^{(k)N}_iG(\tfrac{\cdot}{N})(\bx)\qquad \text{and}\qquad \mathscr B^{(k)N}_{i,j}G(\tfrac{\bx}{N}):=  B^{(k)N}_{i,j}G(\tfrac{\cdot}{N})(\bx)\ .
	\end{equation}

In view of  the definitions of the operators 
		 $\mathscr A^{(k)N}$ and $\mathscr  A^{(k)}$ in \eqref{eq:generatorkN} and \eqref{eq:k-laplacian}, respectively, \eqref{eq:convergence_integral_terms} follows if, for all $i, j \in \{1,\ldots, k\}$ with $i \neq j$, both
		 \begin{equation}\label{eq:integral_terms1}
		 \sup_{t \geq 0}\E^N_{\mu^N}\left[\left(\langle \mathscr A^{(k)N}_iG-\mathscr A^{(k)}_iG,\mathscr X^{(k)N}_t\rangle \right)^2 \right]\underset{N\to \infty}\longrightarrow 0
		 \end{equation}
		 and
		 \begin{equation}\label{eq:integral_terms2}
		 \sup_{t \geq 0}\E^N_{\mu^N}\left[\left(\langle \mathscr B^{(k)N}_{i,j}G, \mathscr X^{(k)N}_t\rangle \right)^2 \right]\underset{N\to \infty}\longrightarrow 0
		 \end{equation}
		 hold true. Let us start with the proof of \eqref{eq:integral_terms1}.
		 
		 \
		 
		 Let us fix $i \in [k]$. By  \cref{lemma:product_fields}, we obtain, for all $N \in \N$ and $t \geq 0$, 
		 \begin{align*}
		 &\ \E^N_{\mu^N}\left[\left(\langle \mathscr A^{(k)N}_iG-\mathscr A^{(k)}_iG,\mathscr X^{(k)N}_t\rangle \right)^2 \right]\\
		 =&\ \sum_{\ell=0}^k \frac{1}{N^{\ell d}}\, \E^N_{\mu^N}\left[\langle \{(\mathscr A^{(k)N}_iG-\mathscr A^{(k)}_iG)\otimes (\mathscr A^{(k)N}_iG-\mathscr A^{(k)}_iG) \}^{(2k-\ell)}, \mathscr X^{(2k-\ell)N}_t\rangle \right]\ .
		 \end{align*}
		 We claim that,  for each $\ell \in [k]_0$, 
		 \begin{equation}\label{eq:vanishing_difference_drift_l}
		 \E^N_{\mu^N}\left[\langle \{(\mathscr A^{(k)N}_iG-\mathscr A^{(k)}_iG)\otimes (\mathscr A^{(k)N}_iG-\mathscr A^{(k)}_iG) \}^{(2k-\ell)}, \mathscr X^{(2k-\ell)N}_t\rangle \right] \underset{N\to \infty}\longrightarrow 0
		 \end{equation}
		 holds. Indeed, let us first observe that the function $F^N:=\mathscr A^{(k)N}_iG-\mathscr A^{(k)}_iG $ is in product form and, if we define
		 \begin{equation*}
		 F^N=\otimes_{i=1}^k\, f^N_i\qquad \text{with}\qquad f^N_j:= \begin{cases}\mathscr A^{(1)N}g_i-\mathscr A^{(1)}g_i &\text{if}\ j=i\\
		 g_j &\text{if}\ j\neq i
		 \ ,
		 \end{cases}
		 \end{equation*} 
		 we have
		 \begin{equation}\label{eq:fN1}
		 \limsup_{N\to \infty}\frac{1}{N^d}\sum_{x_j}|f^N_j(\tfrac{x_j}{N})|<\infty\qquad \text{and}\qquad \limsup_{N\to \infty}\sup_{x_j}|f^N_j(\tfrac{x_j}{N})|<\infty
		 \end{equation}
		 for all $j \in [k]$,	 as well as
		 \begin{equation}\label{eq:fN2}
		 \frac{1}{N^d}\sum_{x_i}|f^N_i(\tfrac{x_i}{N})|\underset{N\to \infty}\longrightarrow 0\qquad \text{and}\qquad \sup_{x_i}|f^N_i(\tfrac{x_i}{N})|\underset{N\to \infty}\longrightarrow 0\ ,
		 \end{equation}
		  the latter being a consequence of the smoothness of $g_i \in \mathscr S^{(1)}$ and the approximation of the Laplacian of $g_i$ by its discrete counterparts.		 
		 By the definitions given in the statement of  \cref{lemma:product_fields}, we have
		 \begin{align*}
		 &\  \left| \E^N_{\mu^N}\left[\langle \{F^N\otimes F^N \}^{(2k-\ell)}, \mathscr X^{(2k-\ell)N}_t\rangle \right]\right|\\
		 =&\ \left|\frac{1}{N^{kd}}\sum_{\bx\in (\Z^d)^k} F^N(\tfrac{\bx}{N}) \frac{1}{N^{kd-\ell d}}\sum_{\by \in (\Z^d)^k} F^N(\tfrac{\by}{N})\, \E^N_{\mu^N}\left[ \{\eta^N_t|(\bx,\by)\}^{(2k-\ell)}\right]\right|\\
		 \leq&\ \scale^{2k-\ell}\cpi^{(2k-\ell)}\sum_{\substack{\mathcal J\subseteq [k]\\
		 |\mathcal J|=\ell}}\sum_{\substack{l:\mathcal J\to [k]\\\oto}} \frac{1}{N^{kd}}\sum_{\bx \in (\Z^d)^k}\left|F^N(\tfrac{\bx}{N}) \right| \frac{1}{N^{kd-\ell d}}\sum_{\by \in (\Z^d)^k}
	 \left|F^N(\tfrac{\by}{N}) \right|	\prod_{j\in J} \1_{\{y_j=x_{l_j}\}}\ ,
		 \end{align*}
		 where in the last inequality we have used the upper bound in \eqref{eq:assumption_bound_timet}. Now, it suffices to observe that, for all $\mathcal J\subseteq [k]$, $|\mathcal J|=\ell$, and one-to-one maps $l:\mathcal J\to [k]$, by the product structure of the function $F^N$ and H\"older's inequality, 
		 \begin{align*}
		 &\ \frac{1}{N^{kd}}\sum_{\bx \in (\Z^d)^k}\left|F^N(\tfrac{\bx}{N}) \right| \frac{1}{N^{k d-\ell d}}\sum_{\by \in (\Z^d)^k}
		 \left|F^N(\tfrac{\by}{N}) \right|	\prod_{j\in J} \1_{\{y_j=x_{l_j}\}}\\
		 \leq&\ \prod_{h=1}^k \left\{\frac{1}{N^d}\sum_{x_h} |f^N_h(\tfrac{x_h}{N})| \right\} \prod_{j \in \mathcal J} \left\{\sup_{x_j} |f^N_j(\tfrac{x_j}{N})| \right\}\prod_{j' \notin \mathcal J}\left\{\frac{1}{N^d}\sum_{y_{j'}} |f^N_{j'}(\tfrac{y_{j'}}{N})| \right\}
		 \end{align*}
		 holds. As a consequence,  \eqref{eq:fN1} and \eqref{eq:fN2} finally yield \eqref{eq:integral_terms1},  all these bounds being independent of $t \geq 0$.

		 \
		 	
		 Let us now turn our attention to  \eqref{eq:integral_terms2} and state the following property for the functions $\mathscr B^{(k)N}_{i,j}G$. 
		 \begin{proposition}\label{proposition:permutationB}
		 	Let $f:(\Z^d)^k\to \R$ be a function growing at most polynomially to infinity and such that, for  $i, j \in [k]$ with $i \neq j$,
		 	\begin{equation*}
		 	f(\bx)=f((ij)\bx)\ ,
		 	\end{equation*}
		 	for all $\bx \in (\Z^d)^k$, 
		 	with $(ij)\in \Sigma_k$ acting $\bx \in (\Z^d)^k$ by permuting  the indices $i$ and $j$. Then, for all functions $G=\otimes_{i=1}^k\, g_i \in \mathscr S^{(k)}$ in product form,
		 	\begin{equation}\label{eq:integration_by_partsB}
		 	\frac{1}{N^{kd}}\sum_{\bx \in (\Z^d)^k} \mathscr B^{(k)N}_{i,j}G(\tfrac{\bx}{N})\, f(\bx)
		 	= \frac{1}{2N^{kd}}\sum_{\bx\in (\Z^d)^k} \left(-\nabla^{(k)N}_{i,j} \nabla^{(k)N}_{j,i}G(\tfrac{\bx}{N})\right)  f(\bx)\ ,
		 	\end{equation} 
		 	where, for all $\bx = (x_1,\ldots, x_k)\in (\Z^d)^k$,	
		 	\begin{equation}\label{eq:double_gradient}
		 	\nabla^{(k)N}_{i,j} G(\tfrac{\bx}{N}):=
		 	\left\{\nabla^{(1)N}_{x_j}g_i(\tfrac{x_i}{N}) \right\} \prod_{\substack{h=1\\h\neq i}}^k g_h(\tfrac{x_h}{N})
		 	:= \left\{N(g_i(\tfrac{x_j}{N})-g_i(\tfrac{x_i}{N}) )\, \1_{\{|x_i-x_j|=1\}}\right\}\prod_{\substack{h=1\\h\neq i}}^k g_h(\tfrac{x_h}{N})\ .	
		 	\end{equation}
		 \end{proposition}
		 \begin{proof}
		 	The claim is an immediate consequence of the definition 
		 	\begin{equation*}
		 	\mathscr B^{(k)N}_{i,j}G(\tfrac{\bx}{N})= N^2\left(g_i(\tfrac{x_j}{N})-g_i(\tfrac{x_i}{N}) \right) \1_{\{|x_i-x_j|=1\}}\prod_{\substack{h=1\\h\neq i}}^k g_h(\tfrac{x_h}{N})\ ,
		 	\end{equation*}
		 	 the permutation invariance of $f: (\Z^d)^k\to \R$ under $(ij)$ and a rearrangement of the summation over $\bx \in (\Z^d)^k$ in \eqref{eq:integration_by_partsB}.
		 \end{proof}
		 In order to prove \eqref{eq:integral_terms2}, as in the proof of \eqref{eq:integral_terms1}, it suffices to show  that, 
		for all $\ell \in [k]_0$, 
		 \begin{align}\nonumber
		 \label{eq:vanishBNl}
		 &\ \E^N_{\mu^N}\left[\langle \{ \mathscr B^{(k)N}_{i,j}G \otimes \mathscr B^{(k)N}_{i,j}G\}^{(2k-\ell)}, \mathscr X^{(2k-\ell)N}_t\rangle \right]\\	
		  =&\ \E^N_{\mu^N}\left[ \frac{1}{N^{kd}}\sum_{\bx \in (\Z^d)^k}\mathscr B^{(k)N}_{i,j}G(\tfrac{\bx}{N})\frac{1}{N^{kd-\ell d}}\sum_{\by \in (\Z^d)^k} \mathscr B^{(k)N}_{i,j}G(\tfrac{\by}{N})\left\{\eta^N_t|(\bx,\by) \right\}^{(2k-\ell)} \right]\ \underset{N\to \infty}\longrightarrow 0\ 
		\end{align}
		holds. To this aim, by \cref{it:permutation1} and \cref{it:permutation2} in  \cref{proposition:properties_product}, we apply  \cref{proposition:permutationB}  twice, yielding
		\begin{align}\label{eq:identityB}
		 &\  \E^N_{\mu^N}\left[\langle \{ \mathscr B^{(k)N}_{i,j}G \otimes \mathscr B^{(k)N}_{i,j}G\}^{(2k-\ell)}, \mathscr X^{(2k-\ell)N}_t\rangle \right]\\
		 =&\ \frac{1}{2N^{kd}}\sum_{\bx\in (\Z^d)^k} \left(-\nabla^{(k)N}_{i,j} \nabla^{(k)N}_{j,i}G(\tfrac{\bx}{N})\right) \frac{1}{2N^{k d-\ell d}}\sum_{\by \in (\Z^d)^k}\left(-\nabla^{(k)N}_{i,j} \nabla^{(k)N}_{j,i}G(\tfrac{\by}{N})\right)  \E^N_{\mu^N}\left[\{\eta^N_t|(\bx,\by)\}^{(2k-\ell)} \right]\ .
		 \nonumber
		\end{align}
		For notational convenience, let us split into the two cases, $\ell=0$ and $\ell \in [k]$. In the first case, $\ell=0$, by \eqref{eq:assumption_bound_timet} and \cref{it:tensor} in  \cref{proposition:properties_product}, we have
		\begin{align*}
		\left|\E^N_{\mu^N}\left[\langle \{ \mathscr B^{(k)N}_{i,j}G \otimes \mathscr B^{(k)N}_{i,j}G\}^{(2k)}, \mathscr X^{(2k-\ell)N}_t\rangle \right] \right|
		\leq&\ \scale^{2k}\cpi^{(2k)} \left(\frac{1}{2N^{kd}}\sum_{\bx \in (\Z^d)^k}\left|\nabla^{(k)N}_{i,j} \nabla^{(k)N}_{j,i}G(\tfrac{\bx}{N}) \right| \right)^2\ ,
		\end{align*}
		which, by 
		\begin{align}\label{eq:bound_dirichlet_form}
		 &\ \frac{1}{2N^{kd}}\sum_{\bx \in (\Z^d)^k}\left|\nabla^{(k)N}_{i,j} \nabla^{(k)N}_{j,i}G(\tfrac{\bx}{N}) \right|\\
		 \nonumber
		\leq&\  \frac{1}{N^d}\left\{\frac{1}{N^d}\sum_{x_i}\sum_{x_j} N^2\left(g_i(\tfrac{x_j}{N})-g_i(\tfrac{x_i}{N})  \right)^2\1_{\{|x_i-x_j|=1\}}\right\}\prod_{\substack{h=1\\h \neq i,j}}^k \left\{\frac{1}{N^d}\sum_{x_h}|g_h(\tfrac{x_h}{N})| \right\}\\
		\nonumber
		+&\ \frac{1}{N^d}\left\{\frac{1}{N^d}\sum_{x_i}\sum_{x_j} N^2\left(g_j(\tfrac{x_j}{N})-g_j(\tfrac{x_i}{N})  \right)^2\1_{\{|x_i-x_j|=1\}}\right\}\prod_{\substack{h=1\\h \neq i,j}}^k \left\{\frac{1}{N^d}\sum_{x_h}|g_h(\tfrac{x_h}{N})| \right\}\\
		\nonumber
		=&\ \frac{1}{N^d} \left\{\frac{1}{N^d}\sum_{x_i}g_i(\tfrac{x_i}{N})(-\tfrac{1}{2}\Delta^{(1)N}g_i)(\tfrac{x_i}{N}) + \frac{1}{N^d}\sum_{x_i}g_j(\tfrac{x_i}{N})(-\tfrac{1}{2}\Delta^{(1)N}g_j)(\tfrac{x_i}{N}) \right\}\prod_{\substack{h=1\\h \neq i,j}}^k \left\{\frac{1}{N^d}\sum_{x_h}|g_h(\tfrac{x_h}{N})| \right\}
		\end{align}
		and the smoothness of the functions $\{g_i: i \in [k]\} \in \mathscr S^{(1)}$, yields \eqref{eq:vanishBNl}. In the second case, $\ell \in [k]$, by recalling the definitions in  \cref{lemma:product_fields}, from	 \eqref{eq:identityB} we get
		\begin{align*}
		 &\  \left|\E^N_{\mu^N}\left[\langle \{ \mathscr B^{(k)N}_{i,j}G \otimes \mathscr B^{(k)N}_{i,j}G\}^{(2k-\ell)}, \mathscr X^{(2k-\ell)N}_t\rangle \right]\right|\\
		 \leq&\ \scale^{2k-\ell}\cpi^{(2k-\ell)}\sum_{\substack{\mathcal J\subseteq [k]\\
		 		|\mathcal J|=\ell}}\sum_{\substack{l:\mathcal J\to [k]\\\oto}} \frac{1}{N^{kd}}\sum_{\bx \in (\Z^d)^k}\left|H^N_{\{i,j\}}(\tfrac{\bx}{N}) \right| \frac{1}{N^{k d-\ell d}}\sum_{\by \in (\Z^d)^k}
		 \left|H^N_{\{i,j\}}(\tfrac{\by}{N}) \right|	\prod_{j'\in J} \1_{\{y_{j'}=x_{l_{j'}}\}}\ ,
		\end{align*}
		where
		\begin{align*}
		H^N_{\{i,j\}}:= \nabla^{(k)N}_{i,j} \nabla^{(k)N}_{j,i}G\ .
		\end{align*}
		For each $\mathcal J\subseteq [k]$, $|\mathcal J|=\ell$, and one-to-one map $l:\mathcal J\to [k]$, let us prove the following:
		\begin{align}\nonumber\label{eq:finalBNij}
		&\ \frac{1}{N^{kd}}\sum_{\bx \in (\Z^d)^k}\left|H^N_{\{i,j\}}(\tfrac{\bx}{N}) \right| \frac{1}{N^{kd-\ell d}}\sum_{\by \in (\Z^d)^k}
		\left|H^N_{\{i,j\}}(\tfrac{\by}{N}) \right|	\prod_{j'\in J} \1_{\{y_{j'}=x_{l_{j'}}\}}\\
		:=&\
		\frac{1}{N^{kd}}\sum_{\bx \in (\Z^d)^k}\left|\nabla^{(k)N}_{i,j} \nabla^{(k)N}_{j,i}G(\tfrac{\bx}{N}) \right| \frac{1}{N^{kd-\ell d}}\sum_{\by \in (\Z^d)^k}
		\left|\nabla^{(k)N}_{i,j} \nabla^{(k)N}_{j,i}G(\tfrac{\by}{N}) \right|	\prod_{j'\in J} \1_{\{y_{j'}=x_{l_{j'}}\}}  \underset{N\to \infty}\longrightarrow 0\ .
		\end{align}
	To this end,  we recall the definition \eqref{eq:double_gradient} and observe   that, despite the fact that the function $H^N_{\{i,j\}}$ is not in product form, the expression in \eqref{eq:finalBNij} further factorizes as follows:
		\begin{equation*}
			\frac{1}{N^{2kd-\ell d}}\sum_{\bx \in (\Z^d)^k}\left|H^N_{\{i,j\}}(\tfrac{\bx}{N}) \right| 	  \prod_{\substack{h=1\\h\neq i,j}}^k\left\{ \sum_{y_h}|g_h(\tfrac{x_h}{N})|\right\} \sum_{y_i}\sum_{y_j} \left|\nabla^{(1)N}_{y_j}g_i(\tfrac{y_i}{N})\nabla^{(1)N}_{y_i}g_j(\tfrac{y_j}{N}) \right|\prod_{j'\in J} \1_{\{y_{j'}=x_{l_{j'}}\}} \ .		
		\end{equation*} 
		In view of this consideration,  the fact that, by the smoothness of the test functions $\{g_i: i \in [k]\}$, we have
		\begin{align*}
		\limsup_{N\to \infty}\sup_{y_h}\sup_{z_h:|z_h-y_h|=1}|\nabla^{(1)N}_{z_h} g_h(\tfrac{y_h}{N})|<\infty
		\end{align*}
		for all $h \in [k]$ and the upper bound in \eqref{eq:bound_dirichlet_form}, we can argue as in the proof of \eqref{eq:integral_terms1} above by using H\"older's inequality. This shows \eqref{eq:finalBNij} and, thus,  concludes the proof of \eqref{eq:integral_terms2}, being all these bounds independent of $t \geq 0$.

\subsection{Proof of \eqref{eq:vanishing_martingale}}\label{section:carre_du_champ} In view of Dynkin's formula \eqref{eq:dynkin} and  the explicit form of the predictable quadratic covariations of the martingale $\{\mathscr M^N_t: t \geq 0\}$, we have, for all $t \geq 0$ and $G \in \mathscr S^{(k)}$, 
\begin{equation}\label{eq:dynkin_martingale}
\E^N_{\mu^N}\left[\left(\langle G, \mathscr M^{(k)N}_t\rangle\right)^2 \right]= \int_0^t \E^N_{\mu^N}\left[\varGamma^{(k)N}_s(G) \right] \dd s\ ,
\end{equation}
where
	\begin{align}\label{eq:carre_du_champ}
\nonumber
\varGamma^{(k)N}_t(G):=&\ 
\mathcal L^N\left(\langle G, \mathscr X^{(k)N}_t\rangle \right)^2 - 2 \langle G, \mathscr X^{(k)N}_t\rangle \, \mathcal L^N \langle G, \mathscr X^{(k)N}_t\rangle\\
=&\ \mathcal L^N\left(\langle G, \mathscr X^{(k)N}_t\rangle \right)^2 - 2 \langle G, \mathscr X^{(k)N}_t\rangle \langle \mathscr A^{(k)N}G, \mathscr X^{(k)N}_t\rangle
\end{align}
is also referred to as the \emph{carr\'{e} du champ} of the higher-order field $\{\mathscr X^{(k)N}_t: t \geq 0\}$ at $G \in \mathscr S^{(k)}$. 	After observing that $	\varGamma^{(k)N}_t(G)\geq 0$ a.s., \eqref{eq:vanishing_martingale} follows from Doob's martingale inequality, \eqref{eq:dynkin_martingale} and the following result.
\begin{lemma}
	\label{lemma:vanishing_martingale}
	For all test functions $G = \otimes_{i=1}^k\, g_i \in \mathscr S^{(k)}$ in product form, we have
	\begin{equation}\label{eq:vanishing_carre}
	\limsup_{N\to \infty} \sup_{t \geq 0} N^d\, \E^N_{\mu^N}\left[\varGamma^{(k)N}_t(G) \right]<\infty\ .	
	\end{equation}
\end{lemma}
The proof of \cref{lemma:vanishing_martingale} 
 relies on the formula for  products of higher-order fields presented in  \cref{lemma:product_fields}.
As an illustration, let us start by computing $\varGamma^{(k)N}_t(G)$ in \eqref{eq:carre_du_champ} for first-order fields  ($k=1$) by using such formula.  Let us fix $G \in \mathscr S^{(1)}$ and $t \geq 0$. Then, for all $N \in \N$, we have 
\begin{align*}
\varGamma^{(1)N}_t(G)=&\ \mathcal L^N\left(\langle G,\mathscr X^{(1)N}_t\rangle \right)^2 - 2 \langle G,\mathscr X^{(1)N}_t\rangle\, \langle \mathscr A^{(1)N} G,\mathscr X^{(1)N}_t\rangle\\
=&\ \langle \mathscr A^{(2)N}\left(G\otimes G \right), \mathscr X^{(2)N}_t\rangle - \langle \mathscr  A^{(1)N} G \otimes G  + G \otimes \mathscr A^{(1)N}G, \mathscr X^{(2)N}_t\rangle\\
+&\ \frac{1}{N^d}\left(\langle \mathscr A^{(1)N}\left(G^2 \right),\mathscr X^{(1)N}_t\rangle - 2\langle  G (\mathscr A^{(1)N}G), \mathscr X^{(1)N}_t\rangle\right)\\
=&\	 \sigma \langle (\mathscr B_{1,2}^{(2)N}+\mathscr B_{2,1}^{(2)N})\left(G \otimes G \right), \mathscr X^{(2)N}_t\rangle + \frac{1}{N^d}\langle \mathscr A^{(1)N}\left( G^2\right) - 2 G (\mathscr A^{(1)N}G), \mathscr X^{(1)N}_t\rangle\ .
\end{align*}
Let us observe that, for all $x, y  \in \Z^d$, 
\begin{align*}
(\mathscr B_{1,2}^{(2)N}+\mathscr B_{2,1}^{(2)N})\left(G\otimes G \right)(\tfrac{x}{N},\tfrac{y}{N})=N^2\left(G(\tfrac{y}{N})-G(\tfrac{x}{N})) \right)^2 \1_{\{|x-y|=1\}} \geq  0
\end{align*}
and
\begin{align*}
\left(\mathscr A^{(1)N}\left( G^2\right) - 2 G \mathscr A^{(1)N}G \right)(\tfrac{x}{N})=\sum_y N^2 \alpha\left(G(\tfrac{y}{N})-G(\tfrac{x}{N}) \right)^2\1_{\{|y-x|=1\}}\geq 0\ ,
\end{align*}
from which, by \eqref{eq:assumption_bound_timet}, we obtain
\begin{align*}
&\ \E^N_{\mu^N}\left[\mathcal L^N\left(\langle G,\mathscr X^{(1)N}_t \rangle\right)^2 - 2 \langle G,\mathscr X^{(1)N}_t\rangle\, \langle \mathscr A^{(1)N} G,\mathscr X^{(1)N}\rangle \right]
\\
\leq&\ 
  \frac{\scale(1+|\sigma|\scale)}{N^d} \left(\frac{1}{N^d}\sum_{x} G(\tfrac{x}{N})\left(-\mathscr A^{(1)N}G(\tfrac{x}{N}) \right)\alpha \right)\ .
\end{align*}

Let us turn now to the general case, $k \geq 1$.
\begin{proofoflemma}{\ref{lemma:vanishing_martingale}}
	Let us apply  \cref{lemma:product_fields}  to obtain, by duality, 
	\begin{align}\label{eq:carre}
	\nonumber
	&\ 	\mathcal L^N\left(\langle G,\mathscr X^{(k)N}_t \rangle\right)^2  - 2\, \langle G,\mathscr X^{(k)N}_t\rangle\, \langle \mathscr A^{(k)N}G,\mathscr X^{(k)N}_t\rangle \\
	=&\ \sum_{\ell=0}^k \frac{1}{N^{\ell d}}\, \langle \mathscr A^{(2k-\ell)N}\{G\otimes G\}^{(2k-\ell)}- 2 \{\mathscr A^{(k)N}G \otimes G\}^{(2k-\ell)}, \mathscr X^{(2k-\ell)N}_t\rangle \ .
	\end{align}
	As for the term
	\begin{align}\label{eq:carre0}
	\nonumber
	&\ \langle \mathscr A^{(2k)N}\{G\otimes G\}^{(2k)}- 2 \{\mathscr A^{(k)N}G \otimes G\}^{(2k)}, \mathscr X^{(2k)N}_t\rangle\\
	=&\ \langle \mathscr A^{(2k)N}(G \otimes G)- \mathscr A^{(k)N}G \otimes G- G \otimes \mathscr A^{(k)N}G, \mathscr X^{(2k)N}_t\rangle 
	\end{align} corresponding to $\ell=0$ on the r.h.s.\ in \eqref{eq:carre}, 
	we have, for all $\bx, \by \in (\Z^d)^k$, 
	\begin{align*}
	&\ \mathscr A^{(2k)N}\left(G \otimes G \right)(\tfrac{\bx\:\by}{N}) - \mathscr A^{(k)N}G(\tfrac{\bx}{N})\, G(\tfrac{\by}{N}) - G(\tfrac{\bx}{N})\, \mathscr A^{(k)N}G(\tfrac{\by}{N})\\
	=&\ \sum_{i=1}^{2k} \mathscr A^{(2k)N}_i\left(G\otimes G \right)(\tfrac{\bx\:\by}{N}) + \sigma \sum_{i=1}^{2k}\sum_{j=1}^{2k} \mathscr B^{(2k)N}_{i,j} \left(G\otimes G \right)(\tfrac{\bx\:\by}{N})\\
	-&\ \sum_{i=1}^k \mathscr A^{(k)N}_i G(\tfrac{\bx}{N})\, G(\tfrac{\by}{N}) - \sum_{i=1}^{k} G(\tfrac{\bx}{N})\,	 \mathscr A^{(k)N}_iG(\tfrac{\by}{N})\\
	-&\ \sigma \sum_{i=1}^k\sum_{j=1}^k \mathscr B^{(k)N}_{i,j}G(\tfrac{\bx}{N})\, G(\tfrac{\by}{N})  - \sigma \sum_{i=1}^k\sum_{j=1}^k G(\tfrac{\bx}{N})\, \mathscr B^{(k)N}_{i,j}G(\tfrac{\by}{N})\\
	=&\ \sigma \sum_{i=1}^k \sum_{j=k+1}^{2k} \mathscr B^{(2k)N}_{i,j}\left(G\otimes G \right)(\tfrac{\bx\:\by}{N})+\sigma \sum_{i=k+1}^{2k}\sum_{j=1}^k\mathscr B^{(2k)N}_{i,j}\left(G\otimes G \right)(\tfrac{\bx\:\by}{N}) \ .
	\end{align*}
	Roughly speaking, of all the expression above only the interaction terms between particles of the $\bx$ group with those of the $\by$ group remains. Then, the  expression  
	\begin{align}\label{eq:carre00}\nonumber
	&\ \langle \mathscr A^{(2k)N}(G \otimes G)- \mathscr A^{(k)N}G \otimes G- G \otimes \mathscr A^{(k)N}G, \mathscr X^{(2k)N}_t\rangle\\
	=&\ \sigma \sum_{i=1}^k \sum_{j=k+1}^{2k} \langle \mathscr B^{(2k)N}_{i,j}(G\otimes G), \mathscr X^{(2k)N}_t\rangle+ \sigma \sum_{i=k+1}^{2k} \sum_{j=1}^k \langle \mathscr B^{(2k)N}_{i,j}(G\otimes G), \mathscr X^{(2k)N}_t\rangle
	\end{align}
	is either identically zero if $\sigma = 0$ or, if $\sigma \in \{-1,1\}$, satisfies
	\begin{align*}
	\limsup_{N\to \infty}\sup_{t \geq 0} N^d\,  \E^N_{\mu^N}\left[\left|	\langle \mathscr A^{(2k)N}(G \otimes G)- \mathscr A^{(k)N}G \otimes G- G \otimes \mathscr A^{(k)N}G, \mathscr X^{(2k)N}_t\rangle \right| \right]<\infty\ .
	\end{align*} 
	Indeed,   the arguments  in the proof of \eqref{eq:integral_terms2} (cf.\  \cref{section:proof_integral_term})   apply also to this case with $k$ replaced by $2k$, because they rely on assumption \ref{it:assumption_bound} of  \cref{theorem:hydrodynamics},  and the product form of the test function $G = \otimes_{i=1}^k\, g_i \in \mathscr S^{(k)}$ only.
	
	Let us show now that each of the remaining terms in \eqref{eq:carre} stays bounded in mean uniformly over time as $N\to \infty$, i.e.\
	\begin{equation}\label{eq:carrel}
	\limsup_{N\to \infty}\sup_{t \geq 0} 
	\E^N_{\mu^N}\left[\left|\langle \mathscr A^{(2k-\ell)N}\{G\otimes G\}^{(2k-\ell)}- 2 \{\mathscr A^{(k)N}G \otimes G\}^{(2k-\ell)}, \mathscr X^{(2k-\ell)N}_t\rangle \right|\right]<\infty
	\end{equation}
	for each $\ell \in [k]$; then, the additional factor $N^{-\ell d}$ will ensure \eqref{eq:vanishing_carre}. As a first step, we observe that
	\begin{align}\label{eq:carrel1}
	\nonumber
	&\ \limsup_{N\to \infty}\sup_{t \geq 0} \E^N_{\mu^N}\left[\left|\langle \mathscr A^{(2k-\ell)N}\{G\otimes G\}^{(2k-\ell)}, \mathscr X^{(2k-\ell)N}_t\rangle \right|\right]\\
	=&\ \limsup_{N\to \infty}\sup_{t \geq 0} \E^N_{\mu^N}\left[\left|\langle \mathscr A^{(2k-\ell)}\{G\otimes G\}^{(2k-\ell)}, \mathscr X^{(2k-\ell)N}_t\rangle \right|\right] < \infty\ ,	
	\end{align}
	because, as a consequence of \eqref{eq:convergence_integral_terms} which still holds with $2k-\ell$ in place of $k$ and $\{G\otimes G\}^{(2k-\ell)} \in \mathscr S^{(2k-\ell)}$ in place of $G\in\mathscr S^{(k)}$, we have
	\begin{equation}\label{eq:important}
	 \limsup_{N\to \infty} \sup_{t \geq 0} \E^N_{\mu^N}\left[\left(\langle \mathscr A^{(2k-\ell)N}\{G \otimes G\}^{(2k-\ell)}-\mathscr A^{(2k-\ell)}\{G\otimes G\}^{(2k-\ell)}, \mathscr X^{(2k-\ell)N}_t\rangle  \right)^2  \right]\underset{N\to \infty}\longrightarrow 0\ .
	\end{equation}
	The upper bound in \eqref{eq:carrel1} follows from \eqref{eq:assumption_bound_timet} and the fact that $\mathscr A^{(2k-\ell)}\{G\otimes G\}^{(2k-\ell)} \in \mathscr S^{(2k-\ell)}$ does not depend on $N \in \N$.
	We are left with the proof that
	\begin{equation}\label{eq:carrel2}
	\limsup_{N\to \infty}\sup_{t \geq 0} \E^N_{\mu^N}\left[\left|\langle \{\mathscr A^{(k)N}G\otimes G\}^{(2k-\ell)}, \mathscr X^{(2k-\ell)N}_t \rangle \right| \right]<\infty
	\end{equation}
	holds. To this aim, in view of the definitions in \cref{lemma:product_fields}, we write
	\begin{align*}
	&\ \langle \{\mathscr A^{(k)N}G\otimes G\}^{(2k-\ell)}, \mathscr X^{(2k-\ell)N}_t \rangle 
	= \frac{1}{N^{kd}}\sum_{\bx \in (\Z^d)^k}\mathscr A^{(k)N}G(\tfrac{\bx}{N}) \frac{1}{N^{kd-\ell d}}\sum_{\by \in (\Z^d)^k} G(\tfrac{\by}{N})\, \{\eta^N_t|(\bx,\by)\}^{(2k-\ell)} \ .
	\end{align*}
	Then, by exploiting the product structure and smoothness of $G \in \mathscr S^{(k)}$ and, thus, arguing as in  \cref{section:proof_integral_term}, we conclude that
	\begin{equation}\label{eq:important2}
	\sup_{t \geq 0}\E^N_{\mu^N}\left[\left(\langle \{\mathscr A^{(k)N}G \otimes G\}^{(2k-\ell)}-\{\mathscr A^{(k)}G \otimes G\}^{(2k-\ell)}, \mathscr X^{(2k-\ell)N}_t\rangle \right)^2 \right]\underset{N\to \infty}\longrightarrow 0\ .
	\end{equation}
	The observation that $\{\mathscr A^{(k)}G \otimes G\}^{(2k-\ell)} \in \mathscr S^{(2k-\ell)}$ is independent of $N \in \N$, combined with \eqref{eq:assumption_bound_timet}, yields \eqref{eq:carrel2}. This concludes the proof of the lemma.		
	\end{proofoflemma}

\section{Proof of \cref{theorem:fluctuations}}\label{section:proof_fluctuations}

The proof of \cref{theorem:fluctuations} employs most of the results developed for the proof of \cref{theorem:hydrodynamics}. In particular, Dynkin's formula implies that, for all $G \in \mathscr S^{(k)}$, $N\in \N$ and $t \geq 0$, 
\begin{align}\label{eq:dynkin_fluctuations}\nonumber
\langle G,  \mathscr Y^{(k,\scale)N}_t\rangle
=&\  \langle G, \mathscr Y^{(k,\scale)N}_0\rangle + \int_0^t \mathcal L^N \langle G, \mathscr Y^{(k,\scale)N}_s\rangle\, \dd s+\langle G, \mathscr M^{(k,\scale)N}_t\rangle 
\\=&\ \langle G, \mathscr Y^{(k,\scale)N}_0\rangle + \int_0^t \mathcal L^N \langle G, \mathscr Y^{(k,\scale)N}_s\rangle\, \dd s + N^{d/2} \langle G, \mathscr M^{(k)N}_t\rangle 
\end{align}
holds, with $\{\mathscr M^{(k)N}_t: t \geq 0\}$ being the same $(\mathscr S^{(k)})'$-valued martingale appearing in \eqref{eq:dynkin}. Because, by \eqref{eq:self-adjoint}, 	
\begin{equation}\label{eq:duality_fluctuations}
\mathcal	 L^N \E^N_{\mu_\scale}\left[\langle G,\mathscr X^{(k)N}_t\rangle  \right] = \E^N_{\mu_\scale}\left[\langle \mathscr A^{(k)N}G, \mathscr X^{(k)N}_t\rangle \right]=0\ ,
\end{equation}
duality again applies, yielding
\begin{align*}
\mathcal L^N \langle G, \mathscr Y^{(k,\scale)N}_t\rangle =&\  N^{d/2} \left(\langle \mathscr A^{(k)N}G, \mathscr X^{(k)N}_t\rangle - \E^N_{\mu_\scale}\left[\langle \mathscr A^{(k)N}G,\mathscr X^{(k)N}_t\rangle  \right] \right)= \langle \mathscr A^{(k)N}G, \mathscr Y^{(k,\scale)N}_t\rangle\ .
\end{align*} 
Therefore, by substituting the above expression into \eqref{eq:dynkin_fluctuations}, we obtain   linear stochastic  equations in $(\mathscr S^{(k)})'$  as governing the evolution of the fluctuation fields.

\subsection{One-time distribution}
As a first step, we characterize the limiting  one-time distributions of $\{\mathscr Y^{(k,\scale)N}_\cdot: N \in \N \}$.  For this purpose,  we first observe that, by definition \eqref{eq:fluctuation_fields2}, 
\begin{equation}\label{eq:centered_one-time_distribution}
\E^N_{\mu_\scale}\left[\langle G, \mathscr Y^{(k,\scale)N}_t\rangle \right]=0
\end{equation}
holds for all $N \in \N$, $t \geq 0$ and $G \in \mathscr S^{(k)}$. In order to characterize its limiting second moments, we need the following analogue of \cref{lemma:product_fields} for the products of expectations of higher-order fields. Recalling the definition of $\pi$ in \eqref{eq:pi_measure} and that
\begin{equation}
	E_{\mu_\scale}\left[[\eta]_\bx \right] = \scale^k\, \pi(\bx)\ ,\qquad \bx \in (\Z^d)^k\ ,
\end{equation} the proof is	 reminiscent to	 that of \cref{lemma:product_fields}; hence, we leave its details to the reader. We recall that $E_{\mu_\scale}$ refers to the expectation w.r.t.\ the measure $\mu_\scale$.
\begin{lemma}\label{lemma:product_expectations}
	Let $k, \ell \in \N$, $\ell \leq k$, and $G \in \mathscr S^{(k)}, H \in \mathscr S^{(\ell)}$. Then, for all $N \in \N$, $\sigma \in \{-1,0,1\}$ and $\scale \in \Scale$, we have
	\begin{align*}
E_{\mu_\scale}\left[\langle G, \mathscr X^{(k)N}\rangle \right] E_{\mu_\scale}\left[\langle H, \mathscr X^{(\ell)N}\rangle \right] = \sum_{h=0}^\ell \frac{\scale^h(-\sigma)^h}{N^{hd}}E_{\mu_\scale}\left[\langle \{G\otimes H\}^{(k+\ell-h)}, \mathscr X^{(k+\ell-h)N}\rangle \right]\ .
	\end{align*}
\end{lemma}
In view of   \cref{lemma:product_fields}, \cref{lemma:product_expectations}, the definition of higher-order fluctuation fields and stationarity of the measure $\mu_\scale$, we obtain, for all $k \in \N$, $t \geq 0$ and $G, H \in \mathscr S^{(k)}$, 
\begin{align}\label{eq:limiting_one-time}\nonumber
&\ \lim_{N\to \infty}  \E^N_{\mu_\scale}\left[\langle G, \mathscr Y^{(k,\scale)N}_t\rangle \langle H, \mathscr Y^{(k,\scale)N}_t\rangle \right]\\
\nonumber
=&\ \lim_{N\to \infty} N^d\, \sum_{\ell=1}^k \frac{1-\scale^\ell(-\sigma)^\ell}{N^{\ell d}}\E^N_{\mu_\scale}\left[\langle \{G\otimes H\}^{(2k-\ell)}, \mathscr X^{(2k-\ell)N}_t\rangle \right]\\
=&\ \lim_{N\to \infty}  (1+\sigma\scale)\, \E^N_{\mu_\scale}\left[\langle \{G\otimes H\}^{(2k-1)}, \mathscr X^{(2k-1)N}_t\rangle \right]
=:	 \llangle G, H\rrangle^{(k)}_{\sigma, \scale}\ ,
\end{align}
which is a finite value independent of $t \geq 0$ and, specifically for product test functions $G=\otimes_{i=1}^k\, g_i \in \mathscr S^{(k)}$, $H=\otimes_{i=1}^k\, h_i\in \mathscr S^{(k)}$, reads as follows:
\begin{equation}\label{eq:variance_gaussian}
\llangle G, H\rrangle^{(k)}_{\sigma,\scale}=	  \sum_{i=1}^k \sum_{j=1}^k \left\{\int_{\R^d}g_i(u) h_j(u)\, \alpha \scale (1+\sigma \scale)\, \dd u \right\} \prod_{\substack{l=1\\l\neq i}}^k \left\{\int_{\R^d} g_l(u)\, \alpha\scale\, \dd u  \right\} \prod_{\substack{l'=1\\l'\neq j}}^k\left\{\int_{\R^d} h_{l'}(u)\, \alpha\scale\, \dd u  \right\}\ .
\end{equation}
We note that, by \eqref{eq:limiting_one-time}--\eqref{eq:variance_gaussian}, the random variables $\{\mathscr Y^{(k,\scale)N}_t: N \in \N\}\subseteq (\mathscr S^{(k)})'$ are tight and all limit points $\mathscr Y^{(k,\scale)}_t \in (\mathscr S^{(k)})$ are centered (cf.\ \eqref{eq:centered_one-time_distribution}) and satisfy, for all $G, H \in \mathscr S^{(k)}$, 
\begin{equation}\label{eq:variance_gaussian2}
\mathscr E^{(k,\scale)}\left[\langle G, \mathscr Y^{(k,\scale)}_t\rangle \langle H, \mathscr Y^{(k,\scale)}_t\rangle \right]= \llangle G, H \rrangle^{(k)}_{\sigma,\scale}\ .
\end{equation}
We conclude the characterization of the limiting one-distribution with the following result.
\begin{lemma}
	For all $k \in \N$ and $t \geq 0$, the sequence $\{\mathscr Y^{(k,\scale)N}_t: N \in \N\}\subseteq (\mathscr S^{(k)})'$ converges in distribution to the unique centered Gaussian distribution $\mathscr Y^{(k,\scale)}_t \in (\mathscr S^{(k)})'$ satisfying \eqref{eq:variance_gaussian2}.
	\end{lemma}
\begin{proof} By stationarity of $\mu_\scale$, we can neglect the time variable.
	 By the product form of the reversible measure $\mu_\scale$, the case $k=1$ is standard (see e.g.\ \cite[Lemma 11.2.1]{kipnis_scaling_1999}). Let us prove the claim for $k \in \N$, $k \geq 2$. For notational convenience, let us consider only product test functions in $\mathscr S^{(k)}$; as we will see, this is not a restriction as the argument we employ applies to all test functions in $\mathscr S^{(k)}$.	Hence, let $G = \otimes_{i=1}^k\, g_i \in \mathscr S^{(k)}$. Then, we get	
	\begin{align}\nonumber
	\label{eq:expansion1}
	\langle G, \mathscr Y^{(k,\scale)N}\rangle=&\ \left\{\langle \widehat G_{\{k\}}, \mathscr X^{(k-1)N}\rangle \langle g_k, \mathscr Y^{(1,\scale)N}\rangle + \langle \widehat G_{\{k\}}, \mathscr Y^{(k-1,\scale)N}\rangle E_{\mu_\scale}\left[\langle g_k, \mathscr X^{(1)N}\rangle \right]\right\}\\
	-&\   \sum_{i=1}^{k-1} \left\{ \varPhi^{(k,\scale)N}_{i,k}(G) + \varPsi^{(k,\scale)N}_{i,k}(G)\right\}\ ,
	\end{align} 
	where, for all $\ell \in [k]$ and distinct  $\{i_1,\ldots, i_\ell\}\subset [k]$, 
	\begin{align*}
	\widehat G_{\{i_1,\ldots, i_\ell\}}:= g_1\otimes\cdots \otimes g_{i_1-1}\otimes g_{i_1+1}\otimes \cdots g_{i_\ell-1}\otimes g_{i_\ell+1}\otimes g_k \in \mathscr S^{(k-\ell)}\ ,
	\end{align*}
	and
	\begin{align*}
	\varPhi^{(k,\scale)N}_{i,k}(G):=&\ \frac{1}{N^{d/2}} \left\{ \frac{1}{N^{kd-d}}\sum_{\bx\in (\Z^d)^{k-1}} \widehat G_{\{k\}}(\tfrac{\bx}{N})\, 	g_k(\tfrac{x_i}{N})\,  [\eta]_{\bx}\right\}\\
	\varPsi^{(k,\scale)N}_{i,k}(G):=&\ \frac{\sigma \scale}{N^{d/2}} \left\{ \frac{1}{N^{kd-d}}\sum_{\bx\in (\Z^d)^{k-1}} \widehat G_{\{k\}}(\tfrac{\bx}{N})\, 	g_k(\tfrac{x_i}{N})\,  \scale^{k-1}\pi(\bx)\right\}\ .	
	\end{align*}
	In particular,  we have, for all $i \in [k-1]$,
	\begin{equation}\label{eq:limit1}
	\lim_{N\to \infty}E_{\mu_\scale}\left[\left(\varPhi^{(k,\scale)N}_{i,k}(G) \right)^2 \right]= \lim_{N\to \infty}E_{\mu_\scale}\left[\left(\varPsi^{(k,\scale)N}_{i,k}(G) \right)^2 \right]=0\ .
	\end{equation}
	as well as
	\begin{align}\label{eq:limit2}
	E_{\mu_\scale}\left[\left(\langle \widehat G_{\{k\}}, \mathscr X^{(k-1)N}\rangle-E_{\mu_\scale}\left[\langle \widehat G_{\{k\}}, \mathscr X^{(k-1)N}\rangle \right] \right)^2 \right]\underset{N\to \infty}\longrightarrow 0\ .
	\end{align}
	By iterating the above argument, we obtain
	\begin{align*}
	\langle G, \mathscr Y^{(k,\scale)N}\rangle=&\ \langle \phi^N(G), \mathscr Y^{(1,\scale)N}\rangle +\varUpsilon^{(k,\scale)N}(G)\ , 	\end{align*}
	with $\phi^N(G) \in \mathscr S^{(1)}$ deterministic and an error term $\varUpsilon^{(k,\scale)N}(G)$ such that	
	\begin{align*}
	\lim_{N\to \infty}E_{\mu_\scale}\left[\left(\varUpsilon^{(k,\scale)N}(G) \right)^2 \right]=0\ .
	\end{align*} 
	The asymptotic Gaussianity of $\mathscr Y^{(1,\scale)N}$ and \eqref{eq:limiting_one-time}--\eqref{eq:variance_gaussian} conclude the proof of the lemma.
\end{proof}

\subsection{Tightness}
As a second step, we show tightness for the sequences	 $\{\mathscr Y^{(k)N}_\cdot|_{[0,T]}: N \in \N\} \subseteq \mathcal D([0,T],(\mathscr S^{(k)})')$, for all $T > 0$.  In view of Mitoma's tightness criterion (\cite[Theorem 4.1]{mitoma_tightness_1983}), it suffices to prove tightness for the sequences $\{\langle G, \mathscr Y^{(k,\scale)N}_\cdot\rangle|_{[0,T]}: N \in \N\} \subseteq \mathcal D([0,T],\R)$, for all $G \in \mathscr S^{(k)}$ and $T > 0$. In view of the decomposition in \eqref{eq:dynkin_fluctuations}, duality \eqref{eq:duality_fluctuations}, the limit in \eqref{eq:limiting_one-time} and stationarity of the measure $\mu_\scale$, we prove, for all $t \geq 0$ and $G \in \mathscr S^{(k)}$ in product form, 
\begin{equation}\label{eq:tightness_fluctuation_1}
\E^N_{\mu_\scale}\left[\left(\langle \mathscr A^{(k)N}G-\mathscr A^{(k)}G, \mathscr Y^{(k,\scale)N}_t\rangle \right)^2 \right]\underset{N\to \infty}\longrightarrow 0
\end{equation}
and
\begin{equation}\label{eq:tightness_fluctuation_2}
\limsup_{N\to \infty}\, N^d\, \E^N_{\mu_\scale}\left[\varGamma^{(k)N}_t(G)  \right]<\infty\ ,	
\end{equation} 
where the \emph{carr\'{e} du champ} $\varGamma^{(k)N}_t(G)$ has been given in \eqref{eq:carre_du_champ}. This ensures tightness for the sequences $\{\langle G, \mathscr Y^{(k,\scale)N}_\cdot|_{[0,T]}\rangle: N \in \N\}$.

\begin{proof}[Proof of \eqref{eq:tightness_fluctuation_1} and \eqref{eq:tightness_fluctuation_2}]
Let us assume without loss of generality that $G \in \mathscr S^{(k)}$ is in product form.
We note that \eqref{eq:tightness_fluctuation_2}  has already been proven in \cref{lemma:vanishing_martingale}.
As for \eqref{eq:tightness_fluctuation_1}, we have	
\begin{align*}
&\ \E^N_{\mu_\scale}\left[\left(\langle \mathscr A^{(k)N}G-\mathscr A^{(k)}G, \mathscr Y^{(k,\scale)N}_t\rangle \right)^2 \right]\\
=&\ 
 N^d \left(\E^N_{\mu_\scale}\left[\left(\langle \mathscr A^{(k)N}G-\mathscr A^{(k)}G, \mathscr X^{(k)N}_t\rangle \right)^2 \right]-\left(\E^N_{\mu_\scale}\left[\langle \mathscr A^{(k)N}G-\mathscr A^{(k)}G, \mathscr X^{(k)N}_t\rangle \right] \right)^2 \right)\\
 =&\ \sum_{\ell=1}^k \frac{1-(-\sigma)^\ell \scale^\ell}{N^{\ell d - d}} \E^N_{\mu_\scale}\left[\langle \{(\mathscr A^{(k)N}G-\mathscr A^{(k)}G)\otimes (\mathscr A^{(k)N}G-\mathscr A^{(k)}G) \}^{(2k-\ell)}, \mathscr X^{(2k-\ell)N}_t\rangle \right]\ ,
\end{align*}
which, by \eqref{eq:vanishing_difference_drift_l}, yields \eqref{eq:tightness_fluctuation_1}.
Now we can apply an argument analogous to the one presented in \cref{appendix:test_functions}  to extend tightness to all  $G\in \mathscr S^{(k)}$. 
\end{proof}

Once the tightness for the sequence $\{\mathscr Y^{(k,\scale)N}_\cdot: N \in \N\}$ is established, we observe that, for all limit points $\{\mathscr Y^{(k,\scale)}_t: t \geq 0\} \in \mathcal D([0,\infty),(\mathscr S^{(k)})')$ and for all $G \in \mathscr S^{(k)}$, we obtain
\begin{align*}
\sup_{t \geq 0}|\langle G, \mathscr Y^{(k,\scale)}_t\rangle - \langle G, \mathscr Y^{(k,\scale)}_{t^-}\rangle|=0\ ,\qquad \text{a.s.}\ ,
\end{align*} 
as a consequence of
\begin{align}\label{eq:vanishing_jumps}
\E^N_{\mu_\scale}\left[\sup_{t \geq 0}|\langle G, \mathscr Y^{(k,\scale)N}_t\rangle-\langle G, \mathscr Y^{(k,\scale)N}_{t^-}\rangle| \right]\underset{N\to \infty}\longrightarrow 0\ ,
\end{align}
ensuring that all limit points are fully supported on $\mathcal C([0,\infty),(\mathscr S^{(k)})')$.

\subsection{Martingales}\label{section:proof_fluctuation_martingales}
 We note that the sequence 
 \begin{equation*}\{\mathscr M^{(k,\scale)N}_\cdot: N \in \N\} \subseteq 	\mathcal D([0,\infty),(\mathscr S^{(k)})')
 \end{equation*} as given in \eqref{eq:dynkin_fluctuations} is tight.
 Moreover,  because limits of uniformly  integrable martingales are martingales (see e.g.\ \cite[Proposition 4.6]{goncalves_universality_2010}) and because of  \eqref{eq:tightness_fluctuation_1},   for any limit point $\{\mathscr Y^{(k,\scale)}_t: t \geq 0\}$, the process 
\begin{equation}\label{eq:martingale_limit}
\{\mathscr M^{(k,\scale)}_t: t \geq 0\} \in \mathscr C([0,\infty),(\mathscr S^{(k)})')
\end{equation} given,  for all $t \geq 0$ and for 	 all $G \in \mathscr S^{(k)}$ in product form, by
\begin{align*}
\langle G, \mathscr M^{(k,\scale)}_t\rangle:= \langle G, \mathscr Y^{(k,\scale)}_t\rangle-\langle G, \mathscr Y^{(k,\scale)}_0\rangle - \int_0^t \langle \mathscr A^{(k)}G, \mathscr Y^{(k,\scale)}_s\rangle\, \dd s\ ,
\end{align*}
is a $(\mathscr S^{(k)})'$-valued martingale w.r.t.\ the natural filtration induced by $\{\mathscr Y^{(k,\scale)}_t: t \geq 0\}$. Then, let us characterize the predictable quadratic variations of such martingales, i.e.\ find, for all $G \in \mathscr S^{(k)}$,  a (predictable)  stochastic process $\{\mathscr V^{(k,\scale)}_t(G) : t \geq 0\} \in \mathcal C([0,\infty),\R)$ such that
\begin{equation}\label{eq:martingaleN}
 \mathscr N^{(k,\scale)}_t(G):=\left(\langle G, \mathscr M^{(k,\scale)}_t\rangle\right)^2-  \mathscr V^{(k,\scale)}_t(G)\ ,\quad t \geq 0\ , 
\end{equation}
is a martingale w.r.t.\ the law and natural filtration of $\{\mathscr Y^{(k,\scale)}_t: t \geq 0\}$. For this purpose, we prove the following lemma.

\begin{lemma} For all $\sigma \in \{-1,0,1\}$, $\scale \in \Scale$, $k \in \N$ and $G \in \mathscr S^{(k)}$, 
	\begin{equation}\label{eq:family_martingalesN}
	\{\mathscr N^{(k,\scale)N}_\cdot(G): N \in \N\}\subseteq \mathcal D([0,\infty),\R)
	\end{equation}
	given, for all $N \in \N$ and $t \geq 0$, by
	\begin{align*}
	\mathscr N^{(k,\scale)N}_t(G):=&\ \left( \langle G, \mathscr M^{(k,\scale)N}_t\rangle\right)^2 - \int_0^t N^d \varGamma^{(k)N}_s(G)\, \dd s\ ,	
	\end{align*}
	is a tight sequence of $\Pr^N_{\mu_\scale}$-integrable martingales. Moreover, 	for all $t \geq 0$, we have
	\begin{align}\label{eq:convergence_carre}
	\E^N_{\mu_\scale}\left[\left(N^d \varGamma^{(k)N}_t(G)- \mathscr U^{(k,\scale)}(G) \right)^2\right]\underset{N\to \infty}\longrightarrow 0\ ,
	\end{align}
	where $\mathscr U^{(k,\scale)}(G)$ is deterministic and defined in \eqref{eq:UK}.	
\end{lemma}
\begin{proof}
	The fact that, for all $N \in \N$, $\{\mathscr N^{(k,\scale)N}_t(G): t \geq 0\}$ is a $\Pr^N_{\mu_\scale}$-integrable martingale is a consequence of Dynkin's formula \eqref{eq:dynkin}. Tightness of the family \eqref{eq:family_martingalesN} follows from tightness of $\{\langle G, \mathscr M^{(k,\scale)N}_\cdot\rangle: N \in \N\}\subseteq \mathcal D([0,\infty),\R)$ (and, hence, $\{(\langle G, \mathscr M^{(k,\scale)N}_\cdot\rangle)^2: N \in \N \}$) and the convergence in \eqref{eq:convergence_carre}. 
	
	Now, let us turn our attention to the proof of \eqref{eq:convergence_carre}. As usual, we prove \eqref{eq:convergence_carre} for product test functions $G=\otimes_{i=1}^k\, g_i \in \mathscr S^{(k)}$ only, and refer to \cref{appendix:test_functions} for the argument which allows the extension of this result to all test functions in $\mathscr S^{(k)}$.

	By  \eqref{eq:carre}, we can decompose $N^d\varGamma^{(k)N}_t(G)$ as the sum of  $k+1$ terms, each of which is indexed by $\ell \in [k]_0$. By \eqref{eq:carrel}, by sending $N\to \infty$,   all terms corresponding to $\ell \in [k]_0 \setminus \{0,1\}$ become negligible; hence, let us consider the two terms corresponding to $\ell=0$ and $\ell =1$. By \eqref{eq:carre00}, \cref{proposition:permutationB}, the stationarity of $\mu_\scale$ and the product form of $G=\otimes_{i=1}^k\, g_i \in \mathscr S^{(k)}$, we have, for all $t \geq 0$ and $i, j \in [k]$,
	\begin{align*}
	\E^N_{\mu_\scale}\left[\left(\sigma \langle \mathscr B^{(2k)N}_{i,j}(G\otimes G)+\mathscr B^{(2k)N}_{j,i}(G\otimes G), \mathscr X^{(2k)N}_t\rangle -  \mathscr U^{(k,\scale)}_{0,\{i,j\}}(G) \right)^2 \right]\underset{N\to \infty}\longrightarrow 0
	\end{align*}
	where
	\begin{align*}
	\mathscr U^{(k,\scale)}_{0,\{i,j\}}(G):=  \left\{\int_{\R^d}  \nabla g_i(u)\cdot  \nabla g_j(u)\, \sigma\alpha^2\scale^2\, \dd u\right\}\prod_{\substack{l=1\\l\neq i}}^k \left\{\int_{\R^d} g_l(u)\, \alpha\scale\, \dd u  \right\} \prod_{\substack{l'=1\\l'\neq j}}^k\left\{\int_{\R^d} g_{l'}(u)\, \alpha\scale\, \dd u  \right\}\ .
	\end{align*}
	By \eqref{eq:important} and \eqref{eq:important2}, similar considerations for the term corresponding to $\ell=1$ and the following  identity 
	\begin{align*}
	\int_{\R^d}  \mathscr A^{(1)} (g_i g_j)(u)- g_i(u)(\mathscr A^{(1)}g_j)(u)-g_j(u)(\mathscr A^{(1)}g_i)(u)\, \alpha \scale\, \dd u
	= \int_{\R^d}  \nabla g_i(u)\cdot  \nabla g_j(u)\, \alpha^2\scale\, \dd u\ ,
	\end{align*} yield the final result.	
\end{proof}
In particular, the Cauchy-Schwarz inequality and \eqref{eq:convergence_carre} imply 
\begin{equation}
\E^N_{\mu_\scale}\left[\left(\int_0^t N^d \varGamma^{(k)N}_s(G)\, \dd s - t\, \mathscr U^{(k,\scale)}(G)\right)^2 \right]\underset{N\to \infty}\longrightarrow 0\ .
\end{equation}
Additionally, an argument as in \cite[p.\ 4171]{franco_phase_2013} ---which, in turn,  employs \cite[Lemma 3]{dittrich1991central} as well as \eqref{eq:vanishing_jumps}---ensures that, for all $G \in \mathscr S^{(k)}$, the martingales $\{\mathscr N^{(k,\scale)N}_t(G): t \geq 0\}$ are uniformly integrable and converge to the martingale in \eqref{eq:martingaleN} with
\begin{equation*}
\mathscr V^{(k,\scale)}_t(G)=t\, \mathscr U^{(k,\scale)}(G)
\end{equation*}
for all $t \geq 0$. This characterizes uniquely  the	distribution of the $(\mathscr S^{(k)})'$-valued martingale in \eqref{eq:martingale_limit}.

\subsection{Uniqueness of limit points}

Finally, since the limit points are fully supported on $\mathcal C([0,\infty),(\mathscr S^{(k)})')$, the convergence at time $t=0$ to the Gaussian random element $\mathscr Y^{(k,\scale)}_0 \in (\mathscr S^{(k)})'$, as well as  Holley-Stroock's theory of generalized Ornstein-Uhlenbeck processes (\cite{holley_generalized_1978}), ensures uniqueness of the limiting process as described in the statement of \cref{theorem:fluctuations}. This concludes the proof of \cref{theorem:fluctuations}.
\appendix

\section{Infinite particle systems}\label{appendix:infinite_particle_systems}
In this appendix, we present a duality-based construction of the infinite  interacting particle systems from  \cref{section:infinite_particle_systems} on the set $\mathcal X	$ of admissible particle configurations given in \eqref{eq:set_admissible}. Our exposition follows closely some of the ideas in
\cite[\S 6]{ayala_hydrodynamic_2016} and \cite[\S 2.2.4]{de_masi_mathematical_1991}. The reader may find alternative constructions for infinite  $\SEP$ ($\sigma = -1$) and $\IRW$ ($\sigma =0$) in the textbooks \cite[\S I.3]{liggett_interacting_2005-1} and \cite[\S 2.2.4]{de_masi_mathematical_1991}, respectively. To this aim, we let $\varLambda \subseteq \Z^d$ denote a finite subset of $\Z^d$, while $\varLambda \nearrow \Z^d$ refers to an increasing sequence of finite subsets  of $\Z^d$  eventually covering the entire $\Z^d$. Moreover, for all $\eta \in \N_0^{\Z^d}$ and $\varLambda \subseteq \Z^d$, we define $\eta^{(\varLambda)}$ as the finite configuration which agrees with $\eta$ in $\varLambda$, and contains no particles outside of $\varLambda$, i.e.\ 
\begin{equation*}
\eta^{(\varLambda)}(x):= \begin{cases}
\eta(x) &\text{if}\ x \in \varLambda\\
0 &\text{otherwise}\ .
\end{cases}
\end{equation*}

We wish to show that, for all $\eta \in \mathcal X \subseteq \N_0^{\Z^d}$ and $t \geq 0$, the following limit
\begin{equation}\label{eq:expectation_limit}
\mathcal E^N_t(\eta|\bx):=\lim_{\varLambda \nearrow \Z^d} \E^N_{\eta^{(\varLambda)}}\left[[\eta^{(\varLambda)N}_t]_\bx \right]
\end{equation}
exists for all $k \in \N$ and $\bx \in (\Z^d)^k$, where $\E^N_{\eta^{(\varLambda)}}$ stands for the expectation w.r.t.\ the law $\Pr^N_{\eta^{(\varLambda)}}$ of the finite particle system starting from $\eta^{(\varLambda)}$ and whose generator  $\mathcal L^N$ is given in \eqref{eq:generator_infinite}. 
Indeed, for each $\varLambda \subseteq \Z^d$,  the process $\{\eta^{N,(\varLambda)}_t: t \geq 0\}$ is a finite particle system in which the total number of particles is conserved by the dynamics. 
Therefore it is a well-defined Markov process. 
Moreover, it is self-dual (see e.g. \cite{carinci_dualities_2015}): according to the discussion at the end of \cref{section:duality}, we have  
\begin{align*}
\E^N_{\eta^{(\varLambda)}}\left[[\eta^{(\varLambda)N}_t]_\bx \right]= \E^N_{\eta^{(\varLambda)}}\left[\frac{[\eta^{(\varLambda)N}_t]_\bx}{\pi(\bx)} \right]\pi(\bx)= \widehat \E^N_\bx\left[\frac{[\eta^{(\varLambda)}]_{\mathbf X^{\bx,N}_t}}{\pi(\mathbf X^{\bx,N}_t)} \right]\pi(\bx) = \sum_{\by \in (\Z^d)^k} \pi(\bx)\, \widehat p^N_t(\bx,\by)\frac{[\eta^{(\varLambda)}]_\by}{\pi(\by)}\ .
\end{align*} Hence, by the definitions of $\eta^{(\varLambda)}$ and $\varLambda \nearrow \Z^d$,   the monotone convergence theorem yields
\begin{equation}\label{eq:finite_expectations}
0\leq \mathcal E^N_t(\eta|\bx)= \sum_{\by \in (\Z^d)^k} \pi(\bx)\, \widehat p^N_t(\bx,\by) \lim_{\varLambda \nearrow \Z^d} \frac{[\eta^{(\varLambda)}]_\by}{\pi(\by)}= \widehat \E^N_\bx\left[\frac{[\eta]_{\mathbf X^{\bx,N}_t}}{\pi(\mathbf X^{\bx,N}_t)} \right] \pi(\bx) <\infty\ ,
\end{equation}
where the last inequality is a consequence of the polynomial (at most) growth  at infinity of the configuration $\eta \in \mathcal X$, and the exponential upper  bound (e.g.\ \cite[Lemma 1.9]{stroock_markov_1997}) for the reversible random walks $\{\mathbf X^{\cdot,N}_t: t \geq 0 \}$ in $(\Z^d)^k$ with   nearest-neighboring jumps and uniformly elliptic rates\footnote{For the exclusion case ($\sigma = -1$), nearest-neighbor jumps for $\{\mathbf X^{\cdot,N}_t: t\geq 0\}$ in $(\Z^d)^k$ corresponding to collision of  particles shall  be replaced by exchange of location for the two nearest-neighboring particles.}. Along the same lines of \cite[pp.\ 13--14]{de_masi_mathematical_1991}, it follows that, for all $\eta \in \mathcal X$ and $t \geq 0$, the limits in \eqref{eq:expectation_limit} correspond to  expectations w.r.t.\ to a unique probability measure, say $p^N_t(\eta,\dd \eta')$, fully supported on $\mathcal X$.
The uniqueness comes from the unique characterization by  its joint factorial moments. Hence, for all functions $f : \mathcal X \to \R$    in the linear span of  bounded functions and those in the form $f(\eta)=[\eta]_\bx$ for some $\bx \in (\Z^d)^k$, we have
\begin{equation}\label{eq:finite_expectations2}
\lim_{\varLambda \nearrow \Z^d} \E^N_{\eta^{(\varLambda)}}\left[f(\eta^{(\varLambda)N}_t) \right] = \int_{\mathcal X} p^N_t(\eta,\dd \eta')\, f(\eta')\ .
\end{equation}
Then for all $\eta \in \mathcal X$,  $k \in \N$ and $\bx \in (\Z^d)^k$, we have
\begin{align*}
 \frac{\dd}{\dd t}\bigg|_{t=0}\int_{\mathcal X} p^N_t(\eta,\dd \eta')\, [\eta']_\bx= A^{(k)N}\left( \frac{[\eta']_\cdot}{\pi(\cdot)} \right)(\bx)\, \pi(\bx) = \mathcal L^N  [\eta]_\bx \ , 	
\end{align*}
where in the first equality we employed \eqref{eq:finite_expectations}--\eqref{eq:finite_expectations2} and the Markovianity of the finite particle system,  while in the second equality we used \eqref{eq:duality_relation}, an identity between two finite summations. 

In conclusion,  by taking the limit $\varLambda \nearrow \Z^d$ in the martingales
\begin{align*}
M^{(\varLambda)N}_t(\eta|\bx):= [\eta^{(\varLambda)N}_t]_\bx-[\eta^{(\varLambda)}]_\bx - \int_0^t \mathcal L^N [\eta^{(\varLambda)N}_s]_\bx\, \dd s
\end{align*}
and
\begin{align*}
\left(M^{(\varLambda)N}_t(\eta|\bx) \right)^2-\int_0^t\, \left\{ \mathcal L^N([\eta^{(\varLambda)N}_s]_\bx)^2 - 2\, [\eta^{(\varLambda)N}_s]_\bx\, \mathcal L^N[\eta^{(\varLambda)N}_s]_\bx\right\}\, \dd s
\end{align*}
and arguing as in \cref{section:proof_fluctuation_martingales}, we obtain the convergence of the uniformly integrable martingales above.
To wit, there exists a unique (by duality) law $\Pr^N_\eta$---with corresponding expectation $\E^N_\eta$---on the Skorokhod space $\mathcal D([0,\infty),\mathcal X)$ of $\mathcal X$-valued \emph{c\`{a}dl\`{a}g} trajectories  such that $\{M^N_t(\eta|\bx): t \geq 0\}$,  given, for all $t \geq 0$, by
\begin{equation*}
M^N_t(\eta|\bx):= [\eta^N_t]_\bx- [\eta]_\bx-\int_0^t \mathcal L^N[\eta^N_s]_\bx\, \dd s\ ,
\end{equation*}
is a martingale w.r.t.\  $\Pr^N_\eta$ such that, a.s.,  $M^N_0(\eta|\bx)=0$; similarly for
\begin{align*}
\left(M^N_t(\eta|\bx) \right)^2- \int_0^t\, \left\{\mathcal L^N([\eta^N_s]_\bx)^2 - 2\, [\eta^N_s]_\bx\, \mathcal L^N[\eta^N_s]_\bx\right\}\, \dd s\ .
\end{align*}

\section{Function spaces}\label{appendix:test_functions}

 	 In this appendix, we first review some basic facts about the function spaces $\mathscr S^{(k)}$ and $(\mathscr S^{(k)})'$, and then show that \eqref{eq:convergence_integral_terms} and \eqref{eq:vanishing_martingale} for product test functions yield \eqref{eq:convergence_probability} for all test functions.

 	  	 We start with some definitions. 
 	 For all $n \in \N_0^d$, we let  $h_n\in \mathscr S^{(1)}\subseteq L^2(\R^d)$ denote  the  $n$th Hermite function defined as, for instance, in \cite[Eq.\ (A.13)]{holley_generalized_1978}, and normalized such that $\| h_n\|_{L^2(\R^d)}=\sqrt{\llangle h_n, h_n\rrangle_{L^2(\R^d)}} =1$.	It is well-known that $\left\{h_n: n \in \N_0^d \right\}$ form an orthonormal basis for $L^2(\R^d)$, and they are the eigenfunctions of the self-adjoint  operator $\mathscr G^{(1)}$, given by $\mathscr G^{(1)}f(u):= -\Delta f(u) + \left(u_1^2+\cdots u_d^2 \right)f(u)$  for all smooth functions $f \in \mathscr S^{(1)}$. In particular, for all $n, m \in \N_0^d$,
 	 	\begin{align*}
 	 	\llangle h_n, \mathscr G^{(1)}\,h_m\rrangle_{L^2(\R^d)} = (2 |n|+d)\, \1_{\{n=m\}}\ ,
 	 	\end{align*}
 	 	where $|n|:= \sum_{\ell=1}^d n_i$.

 	 	 Generalizing this to a tensor product space, we define the $k$th-order Hermite functions $\left\{H_\bn: \bn \in (\N_0^d)^k \right\} \subset\mathscr S^{(k)} \subseteq L^2((\R^d)^k)$, 
 	 	given, for all $\bn=(n_1,\ldots, n_k) \in (\N_0^d)^k$, by
 	 	\begin{align*}
 	 	H_\bn= \otimes_{i=1}^k\, h_{n_i}\ .	
 	 	\end{align*}
 	 	They form an orthonormal basis for $L^2((\R^d)^k)$, and satisfy the following identity: for all $\bn, \bm \in (\N_0^d)^k$, 
 	 	\begin{align*}
 	 	\llangle H_\bn, \mathscr G^{(k)} H_\bm\rrangle_{L^2((\R^d)^k)} := \llangle H_\bn, \oplus_{i=1}^k\, \mathscr G^{(1)}_i H_\bm\rrangle_{L^2((\R^d)^k)} = (2|\bn| + k d)\, \1_{\{\bn=\bm\}}\ ,
 	 	\end{align*}
 	 	where, for $\bn = (n_1,\ldots, n_k) \in (\N_0^d)^k$,
 	 	\begin{equation}
 	 	| \bn| := \sum_{i=1}^k |n_i|\ .
 	 	\end{equation}
 	 If we endow on $\mathscr S^{(k)}$ with the $k$th-order Sobolev norm
 	 	\begin{align*}
 	 	\left\| \mathscr Z\right\|_p:=\left\| \mathscr Z\right\|_{p,(k)}:= \sqrt{\sum_{\bn \in (\N_0^d)^k} (2| \bn|+1)^p\, \left(\llangle H_\bn, \mathscr Z\rrangle_{L^2((\R^d)^k)}\right)^2}\ , \quad p \in \Z\ ,
 	 	\end{align*}
 	 	then 
 	 $\mathcal H^{(k)}_p := \closure{\mathscr S^{(k)}}^{\|\cdot\|_p}$,	
 	 	and for $p > 0$ we have the chain of inclusions
 	 	\begin{align*}
 	 	\mathscr S^{(k)} \subseteq  \ldots \subseteq  \mathcal H^{(k)}_p \ldots \subseteq \mathcal H_0^{(k)} = L^2((\R^d)^k) \subseteq \ldots \subseteq \mathcal H^{(k)}_{-p} \subseteq \ldots \subseteq (\mathscr S^{(k)})'\ .	
 	 	\end{align*}
 	 	In particular, there exists a constant $r = r(d,k)>0$ such that, for all $p > q + r$, the canonical embeddings
 	 $
 	 	\mathcal H^{(k)}_p \hookrightarrow \mathcal H^{(k)}_q
 	 $
 	 	are Hilbert-Schmidt.

\

Let us fix $k \in \N$, and recall the definition of the processes $\left\{\mathscr X^{(k)N}_t: t \geq 0\right\} \subseteq \mathcal D([0,\infty),(\mathscr S^{(k)})')$ from  \cref{section:hydrodynamics}. 
Assume for the moment that for all $T > 0$,  there exists $q \in \N$ such that
 	\begin{align}\label{eq:boundedness_hp_norm}
 	\lim_{M\to\infty} \limsup_{N\to \infty} \Pr^N_{\mu^N}\left(\sup_{t \in [0,T]}\| \mathscr X^{(k)N}_t \|_{-q}  > M\right) = 0\ .
 	\end{align}
 	Then note that, for any deterministic solution $\{\sscale_t: t \geq 0\} \in \mathcal C([0,\infty),(\mathscr S^{(k)})')$ of \eqref{eq:deterministic_heat_equation} and for all $T > 0$,  there exists $r \in \N$ such that
 	\begin{align}\label{eq:boundedeness_hp_norm_deterministic_solution} 	
 	\sup_{t \in [0,T]} \| \sscale_t\|_{-r} < \infty\ .
 	\end{align}
 As a consequence, by \eqref{eq:boundedness_hp_norm} and \eqref{eq:boundedeness_hp_norm_deterministic_solution}, the density in $\mathscr S^{(k)}$ of 
 \begin{equation}
 \mathscr S^{(k)}_\otimes:= \sp\left\{G \in \mathscr S^{(k)}: G=\otimes_{i=1}^k\, g_i \right\},
 \end{equation}
(see e.g.\ \cite[Eqs.\ (A.13)--(A.15)]{holley_generalized_1978})
 and \eqref{eq:convergence_integral_terms}--\eqref{eq:vanishing_martingale} (which, by linearity, hold for all functions in $\mathscr S^{(k)}_\otimes$ defined above), 	we obtain \eqref{eq:convergence_probability} for all test functions $G \in \mathscr S^{(k)}$. 
 
 Thus it remains to prove \eqref{eq:boundedness_hp_norm}.
 	 	In view of Dynkin's decomposition, \eqref{eq:boundedness_hp_norm} follows if 
 	 	\begin{align}\label{eq:bound1}
 	 	\limsup_{N\to \infty} \E^N_{\mu^N}\left[\| \mathscr X^{(k)N}_0\|^2_{-q}\right] < \infty
 	 	\end{align}
 	 	\begin{align}\label{eq:bound2}
 	 	\limsup_{N\to \infty} \E^N_{\mu^N}\left[\sup_{t\in [0,T]}\left\| \int_0^t \mathcal L^N \mathscr X^{(k)N}_s\, \dd s\right\|^2_{-q}\right] &< \infty
 	 	\end{align}
 	 	and
 	 	\begin{align}\label{eq:bound3}
 	 	\limsup_{N\to \infty} \E^N_{\mu^N}\left[\sup_{t\in [0,T]} \|\mathscr M^{(k)N}_t \|^2_{-q} \right] < \infty
 	 	\end{align}
 	 	hold.

 	 	To prove \eqref{eq:bound1}, we use the definition of $\|\cdot\|_{-q}$, the reverse Fatou's lemma, and assumption \ref{it:assumption_bound} in  \cref{theorem:hydrodynamics} to obtain
 	 	\begin{align*}
 	 	\limsup_{N\to \infty} \E^N_{\mu^N}\left[\| \mathscr X^{(k)N}_0\|^2_{-q}\right]\ &\leq\ \sum_{\bn \in (\N_0^d)^k} (2| \bn| +kd)^{-q} \limsup_{N\to \infty} \E^N_{\mu^N}\left[\left(\langle H_\bn, \mathscr X^{(k)N}_0 \rangle \right)^2 \right]\\
 	 	&\leq\ C \sum_{\bn \in (\N_0^d)^k} (2\| \bn\| +kd)^{-q}\,  \left(\| H_\bn\|_{L^1((\R^d)^k)} \right)^2\ ,	
 	 	\end{align*}
 	 	for some constant $C> 0$ independent of $N \in \N$.
 	 	Since the $L^1((\R^d)^k)$ norms of the Hermite functions are uniformly---in $\bn \in (\N_0^d)^k$---bounded by their corresponding $L^2((\R^d)^k)$ norms (see e.g.\ \cite[Theorem 2.1]{larsson-cohn2002}), i.e.\ there exists $c > 0$ such that
 	 	\begin{equation}\label{eq:boundL1norms_hermite}
 	 	\| H_\bn\|_{L^1((\R^d)^k)}\leq c \| H_\bn\|_{L^2((\R^d)^k)} =c\ , \quad \forall\, \bn \in (\N_0^d)^k\ ,
 	 	\end{equation}
 	 	we deduce that \eqref{eq:bound1} holds with $q > k d$. 

 	 	To prove \eqref{eq:bound2}, we proceed analogously and employ the limit statement
 	 	\begin{equation}
 	 	\E^N_{\mu^N}\left[\left(\langle \mathscr A^{(k)N}G-\mathscr A^{(k)}G,\mathscr X^{(k)N}_t \right)^2 \right]\underset{N\to \infty}\longrightarrow 0\ ,
 	 	\end{equation}  whose proof was given in \cref{section:proof_integral_term}.
 	 	Then we use	 the fact that $H_\bn \in \mathscr S^{(k)}_\otimes$ for all $\bn \in (\N_0^d)^k$ to obtain
 	 	\begin{align*}
 	 	&\limsup_{N\to \infty} \E^N_{\mu^N}\left[\sup_{t\in [0,T]}\left\| \int_0^t \mathcal L^N	 \mathscr X^{(k)N}_s\, \dd s\right\|^2_{-q}\right]\\
 	 	\leq&\ \sum_{\bn \in (\N_0^d)^k} (2|\bn| + kd)^{-q} T \limsup_{N\to \infty} \sup_{t \in [0,T]}\E^N_{\mu^N}\left[\left(\langle \mathscr A^{(k)N} H_\bn, \mathscr X^{(k)N}_t \right)^2 \right]\\
 	 	\leq&\ 2 T C\sum_{\bn \in (\N_0^d)^k} (2|\bn| + kd)^{-q} \left(\| \mathscr  A^{(k)} H_\bn\|_{L^1((\R^d)^k)} \right)^2
 	 	\end{align*}	
 	 	for some constant $C > 0$ independent of $N \in \N$.
 	 	By the triangle inequality, known relations for the Hermite functions (see e.g.\ \cite[Eq.\ (A.9)]{holley_generalized_1978}) and \eqref{eq:boundL1norms_hermite}, we have, for some  constant $C>0$ independent of $\bn \in (\N_0^d)^k$, 
 	 	\begin{align*}
 	 	\| \mathcal  A^{(k)} H_\bn\|_{L^1((\R^d)^k)} \leq \sum_{i=1}^k \|\tfrac{\alpha}{2}\Delta_i h_{n_i} \|_{L^1(\R^d)} \leq C\alpha\sum_{i=1}^k  |n_i|  = C		\alpha |\bn|\ .
 	 	\end{align*}
 	 	Hence, \eqref{eq:bound2} holds with $q > k d +2$.	

 	 	Last but not least, an analogous argument employing \cref{lemma:vanishing_martingale}---which, again, holds because the $k$th-order Hermite functions belong to $\mathscr S^{(k)}_\otimes$---yields \eqref{eq:bound3} with $q > kd$.
 		
 	\subsection*{Acknowledgments} 
 	F.S.\ would like to thank Mario Ayala and Frank Redig for useful discussions. 
 	J.P.C.\ acknowledges partial financial support from the US National Science Foundation (DMS-1855604). 
 	F.S.\  was financially supported
 	by the European Union's Horizon 2020 research and innovation programme under the Marie-Sk\l{}odowska-Curie grant
 	agreement No. 754411.
\bibliographystyle{acm}

 \end{document}